\documentclass[a4paper, 11pt, DIV=12]{scrartcl}
\usepackage[utf8]{inputenc}
\usepackage[T1]{fontenc}
\usepackage[english]{babel}
\usepackage{lmodern}
\usepackage{graphicx}
\usepackage{xcolor}
\usepackage{booktabs}
\usepackage[shortlabels]{enumitem}
\usepackage{amsmath}
\usepackage{amsfonts}
\usepackage{amssymb}
\usepackage{amsthm}
\usepackage{bm}

\usepackage{comment}
\usepackage{tikz}
\usetikzlibrary{calc,shapes,decorations.pathmorphing, decorations.markings, patterns,backgrounds}

\usepackage[colorlinks=false, pdfborder={0 0 0}, pdfborderstyle={}]{hyperref}
\usepackage[nameinlink,capitalise,noabbrev]{cleveref}
\crefname{equation}{}{}

\numberwithin{equation}{section}

\usepackage{lineno}
\newcommand*\patchAmsMathEnvironmentForLineno[1]{
  \expandafter\let\csname old#1\expandafter\endcsname\csname #1\endcsname
  \expandafter\let\csname oldend#1\expandafter\endcsname\csname end#1\endcsname
  \renewenvironment{#1}
  {\linenomath\csname old#1\endcsname}
  {\csname oldend#1\endcsname\endlinenomath}}
  \newcommand*\patchBothAmsMathEnvironmentsForLineno[1]{
  \patchAmsMathEnvironmentForLineno{#1}
  \patchAmsMathEnvironmentForLineno{#1*}}
  \AtBeginDocument{
  \patchBothAmsMathEnvironmentsForLineno{equation}
  \patchBothAmsMathEnvironmentsForLineno{align}
  \patchBothAmsMathEnvironmentsForLineno{flalign}
  \patchBothAmsMathEnvironmentsForLineno{alignat}
  \patchBothAmsMathEnvironmentsForLineno{gather}
  \patchBothAmsMathEnvironmentsForLineno{multline}
}

\theoremstyle{definition}

\crefname{question}{Question}{Questions}

\newtheorem{defi}{Definition}[section]
\crefname{defi}{Definition}{Definitions}

\crefname{ex}{Example}{Examples}

\newtheorem{alg}{Algorithm}
\crefname{alg}{Algorithm}{Algorithms}

\theoremstyle{plain}
\newtheorem{thm}[defi]{Theorem}
\crefname{thm}{Theorem}{Theorems}

\newtheorem{step}{Step}
\crefname{step}{Step}{Steps}

\newtheorem{conj}{Conjecture}[section]
\crefname{conj}{Conjecture}{Conjectures}

\newtheorem{lemma}[defi]{Lemma}
\crefname{lemma}{Lemma}{Lemmas}

\crefname{cor}{Corollary}{Corollaries}

\newtheorem{claim}{Claim}
\crefname{claim}{Claim}{Claims}

\newtheorem{prop}[defi]{Proposition}
\crefname{prop}{Proposition}{Propositions}

\crefname{obs}{Observation}{Observations}

\theoremstyle{remark}

\crefname{rmk}{Remark}{Remarks}

\newcommand{\e}{\varepsilon}
\newcommand{\si}{\sigma}
\newcommand{\al}{\alpha}
\newcommand{\be}{\beta}
\newcommand{\de}{\delta}
\newcommand{\De}{\Delta}
\newcommand{\ga}{\gamma}

\newcommand{\N}{\mathbb{N}}

\newcommand{\cA}{\mathcal{A}}
\newcommand{\bA}{\mathbf{A}}
\newcommand{\cB}{\mathcal{B}}
\newcommand{\cC}{\mathcal{C}}
\newcommand{\cD}{\mathcal{D}}

\newcommand{\cF}{\mathcal{F}}

\newcommand{\cH}{\mathcal{H}}

\newcommand{\cJ}{\mathcal{J}}
\newcommand{\cM}{\mathcal{M}}
\newcommand{\cP}{\mathcal{P}}
\newcommand{\cQ}{\mathcal{Q}}
\newcommand{\cR}{\mathcal{R}}
\newcommand{\bR}{\mathbf{R}}

\newcommand{\card}[1]{\left| #1 \right|}
\DeclareMathOperator{\exn}{ex}

\author{Jan Corsten\thanks{Department of Mathematics, LSE, London WC2A 2AE, Email: \texttt{j.corsten@lse.ac.uk}.} \and Tuan Tran\thanks{Discrete Mathematics Group, Institute for Basic Science (IBS), Daejeon, Republic of Korea. E-mail: \texttt{tuantran@ibs.re.kr}. Supported by the Institute for Basic Science (IBS-R029-Y1). Much of the work was done while supported by the Czech Science Foundation, grant number GJ16-07822Y, and with institutional support RVO:67985807, while affiliated with the Institute of Computer Science, Czech Academy of Sciences.}
	}
\date{}

\title{\LARGE Balanced supersaturation for some degenerate hypergraphs}

\begin{document}
\maketitle

\begin{abstract}
  A classical theorem of Simonovits from the 1980s asserts that every graph $G$ satisfying ${e(G) \gg v(G)^{1+1/k}}$ must contain $\gtrsim \left(\frac{e(G)}{v(G)}\right)^{2k}$ copies of $C_{2k}$. Recently, Morris and Saxton established a {\em balanced} version of Simonovits' theorem, showing that such $G$ has $\gtrsim \left(\frac{e(G)}{v(G)}\right)^{2k}$ copies of $C_{2k}$, which are `uniformly distributed' over the edges of $G$. Moreover, they used this result to obtain a sharp bound on the number of $C_{2k}$-free graphs via the method of {\em hypergraph containers}. In this paper, we generalise Morris--Saxton's results for even cycles to $\Theta$-graphs. We also prove analogous results for complete $r$-partite $r$-graphs.
  
\end{abstract}

\section{Introduction}
\subsection{Supersaturation theorems}
The \emph{Tur\'an number} $\exn_r(n,H)$ of an
$r$-graph $H$ is the maximum number of edges in an $n$-vertex $r$-graph which does not contain a copy of $H$.
The Erd\H os-Stone-Simonovits theorem \cite{Erdos1946,Erdos1965} asserts that
\[
\exn_2(n,H) = \left( 1 - \frac{1}{\chi(H) -1} \right)\binom{n}{2} + o(n^2)
\]
for every graph $H$ and therefore asymptotically determines the Tur\'an number of every non-bipartite graph $H$. For bipartite graphs, finding the Tur\'an number is usually very challenging and even their order of magnitude is unknown for most of them. Erd\H os \cite{Erdos-compl-upperbound} further proved that $\exn_r(n,H) = o(n^r)$ if and only if $H$ is an $r$-partite $r$-graph. Similarly as for graphs, not much is known for the Tur\'an number of $r$-partite $r$-graphs.
It is natural to ask now how many copies of $H$ a graph on $n$ vertices with more than $\exn_2(n,H)$ edges must contain. Erd\H os and Simonovits \cite{Erdos1983} observed that for
non-$r$-partite $r$-graphs a simple double-counting argument shows that once we pass the extremal number, we can already find a constant fraction of all copies of $H$ in the complete graph.
This fails in general for $r$-partite $r$-graphs with $r\ge 3$ (see \cite{Sidorenko93}), but Erd\H os and Simonovits \cite{ErdosSimonovits} conjectured that in the graph case (i.e. $r=2$) one can always find a constant fraction of the expected number of copies of $H$ in the random graph with the same number of edges. 

\begin{conj}\label{conj:supersat}
  For every bipartite graph $H$ with $v$ vertices and $e$ edges, there is some $C>0$ so that every graph $G$ with $n$ vertices and $m \geq C \cdot \exn_2 (n,H)$ edges contains
  $\Omega(m^e n^{v-2e})$ copies of $H$.
\end{conj}

So far this conjecture has only been verified for very few graphs. In an unpublished manuscript, Simonovits proved the conjecture for even cycles $C_{2\ell}$ provided that $\exn_2 (n,C_{2\ell}) = \Theta ( n^{1+1/\ell} )$, which is known to be true only for $\ell \in \{2,3,5\}$ (see \cite{KST,Benson,Singleton,BondySimonovits,Wenger}).
Recently, two extensions of this theorem were obtained. One by Morris and Saxton \cite{MorrisSaxton} who proved a balanced version of Simonovits' theorem, which (roughly speaking)  additionally 
guarantees the copies of $C_{2 \ell}$ to be uniformly distributed in the graph. Another one by Jiang and Yepremyan \cite{Jiang2017} who extended Simonovits' theorem to linear cycles in hypergraphs. 
Erd\H os and Simonovits further proved \cref{conj:supersat} for all complete bipartite graphs $K_{s,t}$
with $s \leq t$ and $\exn_2 (n,K_{s,t}) = \Theta(n^{2-1/s})$, which is known to be true if $t$ is large enough in terms of $s$ 
and conjectured to be true for all $t \geq s$ (see \cite{norm_graphs,proj_norm_graphs, KST}).
Morris and Saxton obtained a balanced strengthening of this result as well \cite{MorrisSaxton}.

In this paper we shall extend the results of Morris and Saxton to theta graphs ($\theta_{a,b}$ is the graph consisting of $a$ internally vertex-disjoint paths of length $b$, each with the same endpoints) and complete $r$-partite $r$-graphs. The following two supersaturation results are trivial consequences of our main results (see \cref{subsec:balanced} below).

\begin{thm}\label{thm:supersat-normal-theta}
  For all $a,b \geq 2$, there is some $C>0$ so that every graph $G$ with $n$ vertices and $m \geq C \cdot n^{1 + 1/b}$ edges contains 
  $\Omega(m^{ab} n^{2 -a(b+1)})$ copies of $\theta_{a,b}$.
\end{thm}

\begin{thm}\label{thm:supersat-normal-compl}
  For all $2 \leq a_1 \leq \ldots \leq a_r$, there is some $C>0$ so that every $r$-graph $G$ with $n$ vertices and $m \geq C \cdot n^{r - 1/a_1 \cdots a_{r-1}}$ edges contains 
  $\Omega(m^{a_1 \cdots a_r} n^{a_1 + \ldots + a_r -r \cdot a_1 \cdots a_r})$ copies of $K_{a_1, \ldots, a_r}^{(r)}$.
\end{thm}

Note that $\exn_2(n,\theta_{a,b})=\Theta(n^{1+1/b})$ if $a$ is sufficiently large with respect to $b$ (see \cite{theta_up-bound, conlon-theta}), and that $\exn_r(n,K_{a_1, \ldots, a_r}^{(r)})=\Theta (n^{r-1/a_1 \cdots a_{r-1}})$ if $a_r$ is sufficiently large with respect to $a_1,\ldots, a_{r-1}$ (c.f. \cite{Erdos-compl-upperbound, exnumber-compl-hyper}).
Hence we confirm \cref{conj:supersat} for `most' theta graphs and `most' complete $r$-partite $r$-graphs.

\subsection{Counting \texorpdfstring{$H$}{H}-free subgraphs}
It is a central problem in extremal graph theory to determine the number, $F_r(n,H)$, of $H$-free $r$-graphs on $n$ vertices for a given fixed $r$-graph $H$ and a natural number $n$. We trivially have
\begin{equation}\label{eq:general}
2^{\exn_r(n,H)} \le F_r(n,H) \le \sum_{i \le \exn_r(n,H)}\binom{\binom{n}{r}}{i}= n^{O(\exn_r(n,H))}.
\end{equation}
and all existing results in the area seem to indicate that the lower bound in \eqref{eq:general} is closer to the truth.
The problem of estimating $F_r(n,H)$ is essentially solved for every non-$r$-partite $r$-graph $H$. Indeed, in the graph case, Erd\H os, Frankl and R\"odl \cite{ErdosFranklRodl86} showed
\begin{equation}\label{eq:strong}
F_2(n,H)=2^{(1+o(1))\exn_2(n,H)},
\end{equation}
using Szemer\'edi's regularity lemma. The corresponding result for $r$-graphs was proved by Nagle, R\"odl and Schacht \cite{hypergraph_non-deg} via the hypergraph regularity lemma.

For $r$-partite $r$-graphs on the other hand, the problem seems to be more challenging and much less is known. Morris and Saxton \cite{MorrisSaxton} showed that \cref{eq:strong} does not hold for $C_6$. 
Even the weaker bound $F_r(n,H)=2^{O(\exn_r(n,H))}$ (
a conjecture usually attributed to Erd\H os) has been proven in only a few special cases: for most complete bipartite graphs (see \cite{BaloghSamotijKmm,BaloghSamotijKst}), for cycles of length $\ell \in \{4,6,10\}$ (see \cite{KleitmanWinston,MorrisSaxton}), and for $r$-uniform linear cycles (see \cite{MubayiWang,linear-cycles}). In this paper, we confirm the weaker conjecture for most theta graphs and most complete $r$-partite $r$-graphs. 
More precisely we prove the following results.

\begin{thm}\label{thm:count-theta}
For every $a,b \ge 2$, there are at most $2^{O(n^{1+1/b})}$ $\theta_{a,b}$-free graphs on $n$ vertices and at most $2^{o(n^{1+1/b})}$ of them have $o(n^{1+1/b})$ edges.
\end{thm}

\begin{thm}\label{thm:count-compl}
For all $2 \le a_1 \le \ldots \le a_r$, there are at most $2^{O\left(n^{r-1/(a_1 \cdots a_{r-1})}\right)}$ $K_{a_1, \ldots, a_r}^{(r)}$-free $r$-graphs on $n$ vertices and at most $2^{o\left(n^{r-1/(a_1 \cdots a_{r-1})}\right)}$ of them have $o\left(n^{r-1/(a_1 \cdots a_{r-1})}\right)$ edges.
\end{thm}

 In particular it follows for those $r$-graphs $H$ that there is a positive constant $c=c(H)$ such that asymptotically almost every $H$-free $r$-graph has at least $c \cdot \exn_r(n,H)$ edges. This confirms a special case of a conjecture of Balogh, Bollob\'as and Simonovits \cite{bal-bol-sim_conj} which states that this is true for all bipartite graphs $H$ containing a cycle.

\subsection{Balanced supersaturation theorems}
\label{subsec:balanced}
The hypergraph container method, developed independently by Balogh, Morris and Samotij \cite{BaloghMorrisSamotij_Container}, and Saxton and Thomason \cite{SaxtonThomason_Container}, is one of the most successful recent developments in extremal combinatorics.  In order to apply the method, we have to find a family of copies of $H$ in $G$ that are `evenly distributed' in the following sense.

\begin{defi}[{\cite[Definition~5.5]{MorrisSaxton}}]\label{def:Erdos--Simonovits-good}
Let $\al>0$. An $r$-graph $H$ is called {\em Erd\H{o}s-Simonovits $\al$-good for a function} $m=m(n)$ if there exist positive constants $C$ and $k_0$ such that the following holds. Let $k\ge k_0$, and suppose that $G$ is an $r$-graph with $n$ vertices and $k\cdot m(n)$ edges. Then there exists a non-empty collection $\cH$ of copies of $H$ in $G$, satisfying
	\[
	d_{\cH}(\si) \le \frac{C\cdot \card{\cH}}{k^{(1+\al)(\card{\si}-1)}e(G)} \quad \text{for every $\si\subset E(G)$ with $1\le \card{\si}\le e(H)$,}
	\]
	where $d_{\cH}(\si) := \card{\{ H' \in \cH: \si \subset H'\}} $ denotes the degree of $\sigma$ in $\cH$.
\end{defi}

Morris and Saxton \cite{MorrisSaxton} conjectured that every bipartite graph $H$ is Erd\H{o}s-Simonovits $\al$-good for $m(n)=\exn_2(n,H)$ and some $\al=\al(H) >0$ (the same statement is trivially true for non-bipartite graphs). Furthermore, they expect that the family $\cH$ can be chosen so that it contains (up to a multiplicative factor) as many copies of $H$ as the random graph $G(n,m)$ with $m = k \cdot \exn_2(n,H)$, which leads to a stronger form of \cref{conj:supersat}. Their motivation in making \cref{def:Erdos--Simonovits-good} is the following proposition.

\begin{prop}[{\cite[Proposition~5.6]{MorrisSaxton}}]\label{prop:supersat-count}
  Let $H$ be an $r$-graph and let $\al>0$. If $H$ is Erd\H{o}s-Simonovits $\al$-good for $m(n)$, then the following hold.
	\begin{enumerate}[\normalfont (1)]
		\item There are at most $2^{O(m(n))}$ $H$-free $r$-graphs on $n$ vertices,
		\item The number of $H$-free graphs with $n$ vertices and $o(m(n))$ edges is $2^{o(m(n))}$.
	\end{enumerate} 
\end{prop}

A proof-sketch for a similar result was given in \cite{MorrisSaxton}. For completeness, we provide a full proof of \cref{prop:supersat-count} in \cref{sec:appendix-A}. 

We will extend the ideas from \cite{MorrisSaxton} to prove the following %two
balanced supersaturation theorems, which are the main results of this paper.

\begin{thm}\label{thm:supersat-theta}
For all $ a,b \ge 2 $, there are positive constants $ C, \de$ and $k_0$ such that for all $ k \ge k_0 $ and all graphs $G$ with $n$ vertices and $ kn^{1+1/b}$ edges, there exists a family $ \cH$ of copies of $\theta_{a,b}$ in $G$ so that
	\begin{enumerate}[(i)]
		\item $|\cH| \ge \de k^{ab} n^2 $ and
		\item $d_{\cH}(\si) \le \frac{C \cdot |\cH|}{k^{(1+\al)(|\si|-1)}e(G)}$ for all $ \si \subset E(G)$ with $1 \le |\si| \leq ab$, where $\al=\frac{1}{ab-1}$.
	\end{enumerate} 
\end{thm}

\begin{thm}\label{thm:supersat-compl}
For all $ 2 \le a_1 \le \ldots \le a_r $, there are positive constants $ C, \de$ and $ k_0$ such that for all $ k \ge k_0 $ and all $r$-graphs $G$ with $n$ vertices and $ kn^{r-1/(a_1 \cdots a_{r-1})}$ edges, there exists a family $ \cH$ of copies of $K_{a_1, \ldots, a_r}^{(r)}$ in $G$ so that
	\begin{enumerate}[(i)]
		\item $|\cH| \ge \de k^{a_1\cdots a_r}n^{a_1+\ldots + a_{r-1}} $ and
		\item $d_{\cH}(\si) \le \frac{C \cdot |\cH|}{k^{(1+\al)(|\si|-1)}e(G)}$ for all $ \si \subset E(G)$ with $1 \le |\si| \leq a_1 \cdots a_r$, where $\al=\frac{1}{a_1\cdots a_r-1}$.
	\end{enumerate} 
\end{thm}

We thus confirm Morris' and Saxton's conjecture for most theta graphs and the corresponding statements for hypergraphs for most complete $r$-partite $r$-graphs. 
\cref{thm:count-theta} and \cref{thm:count-compl} follow immediately from \cref{prop:supersat-count} combined with \cref{thm:supersat-theta} and \cref{thm:supersat-compl}.
The subsequent work of Ferber, McKinley and Samotji \cite{Ferber2017} establishes a weaker, but significantly easier to prove, supersaturation result that is still sufficiently strong to derive $F_r(n,H)=2^{\exn_r(n,H})$ for a much larger class of $r$-graphs $H$. However, the result of \cite{Ferber2017} is not strong enough to imply anything non-trivial for the Tur\'an problem in random hypergraphs.

We will prove \cref{thm:supersat-theta} in Section 2, \cref{thm:supersat-compl} in Section 3.
We will use in these sections the slightly informal notation $\e \ll \tilde \e$ if $\e \leq c \cdot \tilde \e$ for a sufficiently small constant $c>0$. 

%%%%%%%%%%%%%%%%%%%%%%%%%%%%%%%%%%%%%%%%%%%%%%%%%%%%%%%%%%%%%%%%%%%%%%%%%%%
%%%%%%%%%%%%%%%%%%%%%%%%%%%%%%%%%%%%%%%%%%%%%%%%%%%%%%%%%%%%%%%%%%%%%%%%%%%

\section{Theta graphs}
For $n,k,j \in \N $ and $\de >0$, let \[ \De^{(j)}(\de,k,n) := \frac{k^{ab-1} \cdot n^{1-1/b}}{\left( \de k^{b/(b-1)} \right)^{j-1}}. \]

\begin{defi}
	 Let $a,b,n,k \in \N$ with $a,b \geq 2$, let $\de >0$ and let $G$ be an $n$-vertex graph with $kn^{1+1/b}$ edges. A collection $ \cH $ of copies of $\theta_{a,b}$ in $G$  is \emph{good} for $(a,b,k,n,\de) $ (or simply good if the parameters are understood) if $d_{\cH}(\si) \le \De^{(\card{\si})}(\de,k,n)$ for every non-empty forest $\si \subset E(G)$.
\end{defi}

The aim of this section is to prove the following theorem.

\begin{thm}\label{thm:supersatcyclefree}
	For all $a,b \geq 2 $, there are some positive constants $k_0$ and $\de$, such that for all $ k \ge k_0$ and all graphs $G$ with $n$ vertices and $ kn^{1+1/b}$ edges, there exists a family $ \cH$ of copies of $\theta_{a,b}$ in $G$ of size $|\cH| \ge \de k^{ab} n^2 $ which is good for ($a,b,k,n,\de$).
\end{thm}

\cref{thm:supersatcyclefree} easily implies \cref{thm:supersat-theta}. Indeed, for every $\sigma \subset E(G)$ with $1 \leq \card \sigma \leq ab$ and $d_\cH(\si) >0$, take a forest $\sigma' \subset \sigma$ of maximal size and note that 
\[d_{\cH}(\si) \le d_{\cH}(\sigma') \leq \De^{(\card{\si'})}(\de,k,n) \leq \frac{k^{ab-1} \cdot n^{1-1/b}}{\left( \de k^{b/(b-1)} \right)^{\card {\sigma'}-1}} \leq \frac{k^{ab-1} \cdot n^{1-1/b}}{\left( \de k^{1 + \al} \right)^{\card{\si}-1}},\] where $ \al = 1/(ab-1)$. 
We remark that the worst case for the last inequality is when $|\sigma|=ab$ and $|\sigma'|=ab-a+1$.
\cref{thm:supersatcyclefree} in turn is an immediate consequence of the following proposition.

\begin{prop}\label{prop:addtheta}
    For all $a,b \ge 2 $, there are some positive constants $k_0$ and $\de >0$ such that for all $ k \ge k_0$ and all graphs $G$ with $n$ vertices and $ kn^{1+1/b}$ edges, the following is true. If $ \cH $ is a collection of copies of $ \theta_{a,b} $ in $ G $ which is good for $(a,b,k,n,\de)$ and $ |\cH| \le \de k^{ab}n^2$, then there exists a copy $H \not \in \cH$ of $ \theta_{a,b} $ such that $ \cH \cup \{H\} $ is good for $(a,b,k,n,\de)$.
\end{prop}

The rest of this section is devoted to the proof of \cref{prop:addtheta}.

\subsection{The setup}\label{section:setup}
We define all constants here and fix the important parameters.
Let $a,b \geq 2$ and set $ K = 5ab$, $\e(b) = 1/K^3$, $\e(t-1) = \e(t)^t $ for each $ 2 \le t \le b $, $\de = \e(1)^{2ab+2} $ and $ k_0 = 1/\de $.
Let $n,k \in \N$ with $ k \ge k_0$, and fix a graph $G$ with $n$ vertices and $kn^{1+1/b} $ edges.
Also fix a good collection $ \cH$ of copies of $ \theta_{a,b} $ in $ G $ with $\card{\cH} \le \de k^{ab}n^2$.

We will make the following further assumptions on $G$.
Since $ \de = \e(1)^{2ab+2} $, there are at most 
\[ 
\frac{ab \cdot |\cH|}{\e(1)^{2ab+1} k^{ab-1} n^{1-1/b}} \le ab\cdot \e(1)\cdot e(G)  \ll e(G) 
\] 
edges $ e \in G $ with $d_{\cH}(e) \ge \e(1)^{2ab+1} k ^{ab-1} n^{1-1/b}$.
By deleting all such edges we may assume
\begin{equation}
d_{\cH}(e) < \e(1)^{2ab+1} k ^{ab-1} n^{1-1/b} \text{ for every } e \in E(G) \label{eq:degreebound1}
\end{equation}
(at the cost of slightly weaker constants). In particular, we have
\begin{equation}
d_{\cH}(e) < \De^{(1)}(\de,k,n) \text{ for every } e \in E(G).\label{eq:degreebound2}
\end{equation}

Similarly, since there are at most $K\e(b)kn^{1+1/b} \ll e(G) $ edges incident to vertices of degree at most $K\e(b)kn^{1/b}$, we may assume that
\begin{equation}
	\de(G) \ge K\e(b)kn^{1/b}.
\end{equation}
Finally, we define \emph{saturated} sets of edges.
\begin{defi}[Saturated sets of edges]
Given a non-empty forest $\si \subset E(G) $, we say that $ \si $ is \emph{saturated} if $ d_{\cH}(\si ) \ge \left\lfloor \De^{(\card{\si})}(\de,k,n) \right\rfloor $. Let 
\[ \cF = \left\{ \si \subset E(G) : \si \text{ is saturated} \right\}
\]
denote the collection of all saturated sets of edges.
\end{defi}
We emphasize that in all further results $G$, $\cH$, $\cF$ and all parameters are fixed as above.

\subsection{Preliminaries}

For $ S \subset E(G) $ and $ j \in \N $, define the \emph{$j$-link} of $S$ as \[L_\cF^{(j)}(S) := \left\{ \si \subset E(G) \setminus S : |\si| = j \text{ and } \si \cup \tau \in \cF \text{ for some non-empty } \tau \subset S \right\},\] and let $ L_\cF(S) = \bigcup_{j\ge 1} L_\cF^{(j)}(S)$. We have the following important bound on its size.

\begin{lemma}\label{lemma:linkbound}
For every $ j \in \N $ and every $ S \subset E(G) $, we have 
\[
|L_\cF^{(j)} (S)| \le 2^{ab +\card{S}+1} \cdot \left( \de k^{b/(b-1)} \right)^j .
\]
\end{lemma}

\begin{proof}
		For each non-empty forest $\tau \subset S$, set
		\[
		\cJ(\tau)=\{ \si \subset  E(G) \setminus S: \card \si = j \ \text{and} \ \si \cup \tau \in \cF \}.
		\]
		By the handshaking lemma and the definition of goodness, we obtain
		\[
		\frac{1}{2^{ab}}\cdot\sum_{\si \in \cJ(\tau)}d_{\cH}(\si \cup \tau) \le d_\cH(\tau) \le \De^{(\card \tau)}(\de,k,n),
		\]
		as each edge of $\cH$ is counted at most $2^{ab}$ times in the sum.
		Moreover,
		\[
		\sum_{\si \in \cJ(\tau)}d_{\cH}(\si \cup \tau) \ge \card{\cJ(\tau)}\cdot \lfloor \De^{(\card \tau +j)}(\de,k,n) \rfloor,
		\]
		by the definition of $\cJ(\tau)$ and $\cF$. Hence
		\[
		\card{\cJ(\tau)} \le 
		2^{ab} \cdot \frac{\De^{(\card \tau)}(\de,k,n)}{\lfloor \De^{(\card \tau +j)}(\de,k,n) \rfloor} \le 2^{ab+1}\cdot (\de k^{b/(b-1)})^j.
		\]
		Finally, since the sets $\cJ(\tau)$ cover $L_\cF^{(j)} (S)$, we find that
		\[
		|L_\cF^{(j)} (S)|\le \sum_{\tau}\card{\cJ(\tau)} \le 2^{ab+S+1} \cdot( \de k^{b/(b-1)} )^j, 
		\]
		as desired.
\end{proof}

The following definition and theorem summarise a series of results of Morris and Saxton (see \cite[Section 3]{MorrisSaxton}) which we will use in a similar way to build copies of $\theta_{a,b}$.

\begin{defi}
  Let $ x \in V(G) $ and $2 \le t \in \N$. A \emph{$t$-neighbourhood} of $x$ is a pair $(\cA,\cP)$, in which 
  \begin{itemize}
  \item $\cA=(A_0,A_1,\ldots,A_t)$ is a collection of (not necessarily disjoint) sets of vertices of $G$ with $ A_0 = \{x\} $,
  \item $\cP$ is a collection of paths in $G$ of the form $(x,u_1,\ldots,u_t)$, with $u_i\in A_i$ for each $i\in [t]$.
  \end{itemize}
\end{defi}

For any collection $\cP$ of paths in $G$ and any two vertices $u,v \in V(G)$, let
\[\cP[u \to v]:=\{(x_1,\ldots,x_s) : x_1 = u, x_s = v \} \] 
denote the set of paths in $\cP$ which begin at $ u $ and end at $v$.

\begin{thm}[Morris--Saxton \cite{MorrisSaxton}]\label{thm:refined}
    Given $G,\cH,\cF $ and all constants as in \cref{section:setup},
    there exist $t \in \{2,\ldots,b\}$ and some vertex $x \in V(G)$, for which there is a $t$-neighbourhood $(\cB=\left(B_0,\ldots,B_t),\cQ\right)$ of $x$ with the following seven properties:
    \begin{enumerate}[label=\normalfont{(P\arabic*)}]
    \item \label{P:size} $ \card{B_1} \le kn^{1/b}$ and $ \card{B_t} \le k^{(b-t)/(b-1)}n^{t/b}$.
    \item \label{P:fwneighb} For every $i \in \{0,1,\ldots,t-1\}$ and every $ u \in B_i$, 
    \[
    \card{N(u) \cap B_{i+1}} \ge \e(t)kn^{1/b}.
    \]
    \item \label{P:bwneighb} For every $ v \in B_t $, 
    \[
    \card{N(v) \cap B_{t-1}} \ge \e(t)^2 k^{b/(b-1)}.
    \]
    \item \label{P:manypaths} For every $ v \in B_t $, 
    \[
    \card{\cQ[x \to v]} \ge \e(t)^t k^{(t-1)b/(b-1)}.
    \]
    \item \label{P:avoidF} $\cQ$ avoids $\cF$, i.e.\ $ \sigma \not \subset Q$ for every $\sigma \in \cF$ and every $Q \in \cQ$.
    \item \label{lemma:bal-paths-cont-vertex} For every $w \in B_t$ and $v \in V(G) \setminus \{x,w\} $, there are at most $ b k^{(t-2)b/(b-1)} $ paths $ Q \in \cQ[x \to w] $ containing $v$.
    \item \label{lemma:bal-paths-cont-edges} For every $ \si \subset E(G)$ with $ |\si| \le t-1 $ and every $ w \in B_t$, there are at most $ t^t \cdot k^{(t - |\si| - 1) b/(b-1)} $ paths $ Q \in \cQ[x \to w] $ with $\si \subset E(P)$. 
    \end{enumerate}
\end{thm}
 
 We shall call $(\cB,\cQ)$ a {\em refined $t$-neighbourhood} of $x$. Property \cref{P:avoidF} is slightly different here but completely analogous (in the proof of Lemma 3.6 in \cite{MorrisSaxton}, we need to use \cref{lemma:linkbound} instead of the corresponding lemma in \cite{MorrisSaxton}). 

\subsection{Finding \texorpdfstring{$\theta_{a,b}$}{theta} in refined \texorpdfstring{$t$}{t}-neighbourhoods}
Let $G,\cH, \cF$ and all constants be as in \cref{section:setup} and let $(\cB,\cQ)$ be the refined $t$-neighbourhood for some $x \in V(G)$ and $t \in \{2,\ldots,b\}$ guaranteed by \cref{thm:refined}.

For technical reasons fix \[ X_i(u) \subset N(u) \cap B_{i+1} \text{ of size } |X_i(u)|=\e(t)kn^{1/b} \] for each $i \in [t-1]$ and $u \in B_i$, and \[ X_t(u) \subset N(u) \cap B_{t-1} \text{ of size } |X_t(u)|=\e(t)^2k^{b/(b-1)} \] for each $ u \in B_t$. Furthermore, fix a subset \[ \cQ(z) \subset \cQ (x \to z)  \text{ of size } |\cQ(z) | = \e(t)^t k^{(t-1)b/(b-1)} \] for every $ z \in B_t $.

Using the following algorithm, we shall create many copies $ H $ of $ \theta_{a,b}$ in $ G $ such that $ \cH \cup \{H\} $ is good and deduce that one of them must not be contained in $ \cH$ already. 

\begin{alg}\label{alg}
Initially, let $ \Theta := \emptyset $. As long as possible generate new copies of $ \theta_{a,b}$ and add them to $ \Theta $ via the following process. To create a copy of $\theta_{a,b}$ we shall add edges and denote the subgraph of $G$ induced by the currently selected edges by $H=(V,E)$. (Note that $H,V$ and $E$ are constantly changing.)
\begin{enumerate}
    \item Generate a path $ P_1 = (x=p^{1}_{0}, p^{1}_{1}, \ldots, p^{1}_{t}) $ as follows. For $i=0,1,\ldots,t-1$, choose $p^{1}_{i+1}$ from $X_{i}(p^{1}_{i}) \subset N(p^{1}_{i})\cap B_{i+1}$ such that 
    \[
    p^{1}_{i+1} \notin V \quad \text{and} \quad \{p^{1}_{i},p^{1}_{i+1}\}\not \in L^{(1)}_\cF(E).
    \]
    \item Create a path $Z_1 = (p^{1}_{t}=z^{1}_{0}, z^{1}_{1}, \ldots, z^{1}_{b-t}=:y)$ as follows. Define
    \begin{equation*}
    r(i)=
    \begin{cases}
    t-1 & \text{if } 0 \le i \le b-t \ \text{and $i$ is even},\\
    t & \text{if }  0 \le i \le b-t \ \text{and $i$ is odd}.
    \end{cases}
    \end{equation*}
    For $i = 0, \ldots, b-t-1$, select $ z^{1}_{i+1}$ from $X_{r(i)}(z^{1}_{i})\subset N(z^{1}_{i})\cap B_{r(i+1)}$ such that 
    \[
    z^{1}_{i+1}\notin V \quad \text{and} \quad \{z^{1}_{i},z^{1}_{i+1}\} \not \in L^{(1)}_{\cF}(E).
    \]
    \item For $j = 2, \ldots, a$, create a path $ Z_j = (y= z^{j}_{0}, z^{j}_{1}, \ldots, z^{j}_{b-t}=:z_j) $ as follows. Let
    \begin{equation*}
    	s(i)=
    	\begin{cases}
    		t-1 & \text{if } 0 \le i \le b-t \ \text{and $i+(b-t)$ is even},\\
    		t  & \text{if } 0 \le i \le b-t \ \text{and $i+(b-t)$ is odd}.
    	\end{cases}
    \end{equation*}
    Now, for $i=0,\ldots,b-t-1$, choose $ z^{j}_{i+1}$ from $X_{s(i)}(z^{j}_{i}) \subset N(z^{j}_{i})\cap B_{s(i+1)}$ with 
    \[
    z^{j}_{i+1} \notin V \quad \text{and} \quad \{z^{j}_{i},z^{j}_{i+1}\} \notin L^{(1)}_{\cF}(E).
    \]	
    \item For $j=2, \ldots,a $, pick a path $ P_j \in \cQ(z_j)$ which uses no vertex of $V \setminus \{z_j\}$ and avoids $L_\cF(E)$.\\
    Join the paths $P_1,Z_1,\ldots,P_a,Z_a$ to form a copy of $ \theta_{a,b}$, and add this to $\Theta$. 
\end{enumerate}
\end{alg}

\begin{figure}[ht]
	\centering
	%%%%%%%%%%%%%%%%%%%%%%%tikz%%%%%%%
	\begin{tikzpicture}
	\edef \height {2cm}
	\edef \width  {0.4 * \height}
	\edef \incr {0.5* \height}
	\edef \vxdiam {0.1cm}
	\tikzset{vxstyle/.style={draw=black, fill = black!50, circle, inner sep=0cm, minimum width = \vxdiam}}
	\tikzset{potatostyle/.style={draw=black, ellipse, outer sep = 0cm, inner sep = 0cm, minimum width = \width, minimum height = \height}}
	\tikzset{->-/.style={decoration={
				markings,
				mark=at position #1 with {\arrow{>}}},postaction={decorate}}}
	
	\foreach \x in {1,2}{\coordinate (P\x) at ({(\x-1)*3*\width+ \width},0);}
	\foreach \x in {3,4}{\coordinate (P\x) at ({(\x+1)*3*\width+\width},0);}
	
	\node[vxstyle,label=below:$x$] (x) at (0,0) {};
	\node[vxstyle,label=below:$y$] (y) at (P3) {};
	
	\node[potatostyle, label=below:$B_1$] (B1) at (P1) {};
	\node[potatostyle, label=below:$B_2$, minimum height = {\height+\incr} ] (B1) at (P2) {};
	\node[potatostyle, label=below:$B_{t-1}$, minimum height = {\height+3*\incr}] (B1) at (P3) {};
	\node[potatostyle, label=below:$B_t$, minimum height = {\height+4*\incr}] (B1) at (P4) {};
	
	\coordinate (M) at ($(P2)!0.5!(P3)$);
	\foreach \x in {-1,0,1}{\draw[fill=black, draw=black] ($(M)+(\x*0.3cm,0)$) circle (0.05cm);}
	
	\if11
	\foreach \i in {-1,0,1}{
		\node[vxstyle] (v1\i) at ($(P1)+(0,\i*\height*0.4)$) {};
		\node[vxstyle] (v2\i) at ($(P2)+(0,{\i*(\height+\incr)*0.4})$) {};}
	
	\foreach \i in {-5,-3,2,3,4,5}{\node[vxstyle] (v3\i) at ($(P3)+(0,{\i*(\height+3*\incr)*0.09})$) {};}
	\foreach \i in {-5,-2,1,2,3,5}{\node[vxstyle] (v4\i) at ($(P4)+(0,{\i*(\height+3*\incr)*0.09})$) {};}

	\draw[blue,thick,decoration={markings,
		mark=at position 0.07 with {\arrow{>}},
		mark=at position 0.2 with {\arrow{>}},
		mark=at position 0.67 with {\arrow{>}},
		mark=at position 0.95 with {\arrow{>}}
	}, postaction={decorate}
	] (x) -- (v1-1) -- (v2-1) -- (v3-5) -- (v4-5);
	
	\draw[blue,thick,dotted,decoration={markings,
		mark=at position 0.25 with {\arrow{>}},
		mark=at position 0.55 with {\arrow{>}},
		mark=at position 0.88 with {\arrow{>}}
	}, postaction={decorate}
	] (v4-5) -- (v3-3) -- (v4-2) -- (y);
	
	\draw[red, thick,
	decoration={markings,
		mark=at position 0.02 with {\arrow{<}},
		mark=at position 0.15 with {\arrow{<}},
		mark=at position 0.4 with {\arrow{<}},
		mark=at position 0.88 with {\arrow{<}}
	}, postaction={decorate}
	] (x) -- (v10) -- (v20) -- (v33) -- (v43);
	
	\draw[red,thick,dotted,decoration={markings,
		mark=at position 0.15 with {\arrow{<}},
		mark=at position 0.44 with {\arrow{<}},
		mark=at position 0.82 with {\arrow{<}}
	}, postaction={decorate}
	] (v43) -- (v32) -- (v41) -- (y);
	
	\draw[green,thick,decoration={markings,
		mark=at position 0.025 with {\arrow{<}},
		mark=at position 0.17 with {\arrow{<}},
		mark=at position 0.43 with {\arrow{<}},
		mark=at position 0.88 with {\arrow{<}}
	}, postaction={decorate}
	] (x) -- (v11) -- (v21) -- (v35) -- (v45);
	
	\draw[green,thick,dotted,decoration={markings,
		mark=at position 0.1 with {\arrow{<}},
		mark=at position 0.4 with {\arrow{<}},
		mark=at position 0.81 with {\arrow{<}}
	}, postaction={decorate}
	] (v45) -- (v34) -- (v42) -- (y);
	
	\coordinate (M) at ($(P2)!0.5!(P3)$);
	\foreach \x in {-1,0,1}{\draw[fill=black, draw=black] ($(M)+(\x*0.3cm,0)$) circle (0.05cm);}
	\fi

	\node[draw=black,thick,rounded corners=2pt,above right=2mm] at ($(current bounding box.south west)-(0cm,4mm)$) {%
		\begin{tabular}{llll}
		\raisebox{2pt}{\tikz{\draw[blue, thick] (0,0) -- (5mm,0);}}&$P_1$&\raisebox{2pt}{\tikz{\draw[blue, thick, dotted] (0,0) -- (5mm,0);}}&$Z_1$\\
		\raisebox{2pt}{\tikz{\draw[red, thick] (0,0) -- (5mm,0);}}&$P_2$&\raisebox{2pt}{\tikz{\draw[red, thick, dotted] (0,0) -- (5mm,0);}}&$Z_2$\\
		\raisebox{2pt}{\tikz{\draw[green, thick] (0,0) -- (5mm,0);}}&$P_3$&\raisebox{2pt}{\tikz{\draw[green, thick, dotted] (0,0) -- (5mm,0);}}&$Z_3$\\
		\end{tabular}};
	
	\end{tikzpicture}
	%%%%%%%%%%%%%%%%%%%%%%%%%%%%%%%%%%%%%%%%%%
	\caption{A copy of $\theta_{3,b}$ produced by \cref{alg}.}
	\label{fig:algorithm}
\end{figure}
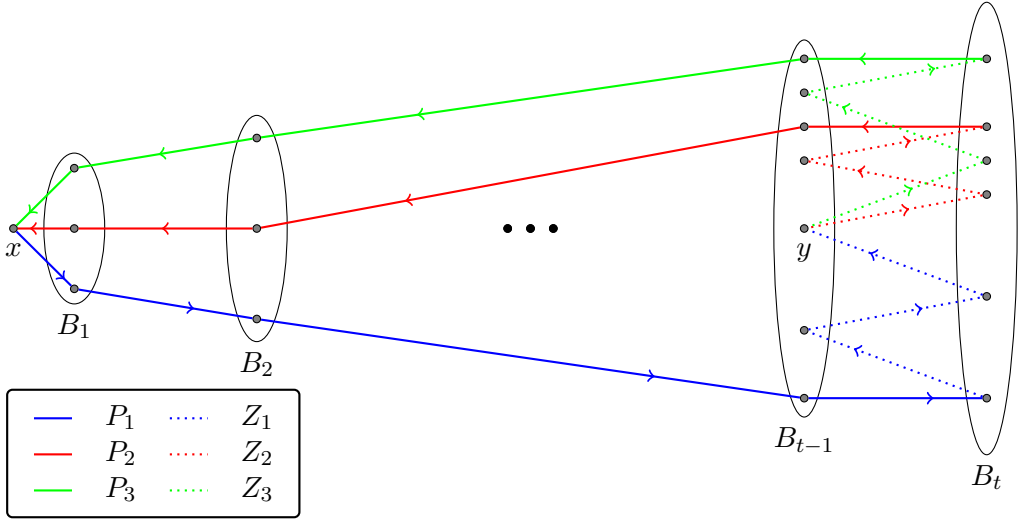

See \cref{fig:algorithm} for an illustration of \cref{alg}. We shall show later that $\card{\Theta}$ is quite large.

\begin{claim} \label{claim:R0}
$\card{\Theta} \ge \e(1)^{2ab}k^{ab}n$.
\end{claim}

Before we proceed with the proof of Claim \ref{claim:R0}, we show how it implies \cref{prop:addtheta}.

\begin{proof}[Proof of \cref{prop:addtheta}]
Since $\card{B_1} \le kn^{1/b}$ by property \ref{P:size} of \cref{thm:refined}, we have  
\[ 
\card{\Theta \cap \cH} \le kn^{1/b}\cdot \max_{e\in E(G)}d_{\cH}(e) \overset{\cref{eq:degreebound1}}{\le} kn^{1/b} \cdot \e(1)^{2ab+1} k ^{ab-1} n^{1-1/b} = \e(1)^{2ab+1} k^{ab} n. 
\] 
Hence, by \cref{claim:R0}, there exists some $ H \in \Theta \setminus \cH $. By the construction of $ \Theta $, the collection $\cH \cup \{H\} $ is good. This finishes our proof.
\end{proof}

The rest of this section is devoted to the proof of \cref{claim:R0}. To make the counting of $\card{\Theta}$ easier to follow, we introduce some notation here. For $i \in [a]$, let $\cR_i$ be the set of all possible choices for the paths $P_1, Z_1,\ldots, Z_a, P_2, \ldots, P_i$ in \cref{alg}. For $\bR_i=(P_1, Z_1,\ldots, Z_a, P_2, \ldots, P_i) \in \cR_i$, let $\cR_{i+1}(\bR_i) := \left\{ P_{i+1} \in Q(z_{i+1}) : (\bR_i,P_{i+1}) \in \cR_{i+1} \right\}$.
We call a vertex $v\in V(\bR_1)\setminus \{x\}$ {\em forward} if either $v\in V(P_1)$ or $v\in B_t$, {\em backward} if $v\in V(Z_2\cup \ldots \cup Z_a)\cap B_{t-1}$.
Hence we have partitioned $V(\bR_1)\setminus\{x\}$ into forward and backward vertices. Let $r_{fw}$ and $r_{bw}$ denote the number of forward and backward vertices respectively of some $\bR_1 \in \cR_1$, and let $ r = r_{fw}+r_{bw}$. It is not difficult to see that $r = ab - (a-1)t$, and
\begin{equation}\label{eq:r}
\begin{cases}
r_{fw}=t +a(b-t)/2, \ r_{bw}=a(b-t)/2 & \text{if } b-t \text{ is even},\\
r_{fw}=t + a(b-t+1)/2 -1, \ r_{bw}=a(b-t-1)/2+1& \text{if } b-t \text{ is odd}.
\end{cases}
\end{equation}

To prove \cref{claim:R0}, we first bound the number of graphs $(P_1,Z_1,\ldots,Z_a)$, chosen in Steps 1--3, in terms of $r_{fw}$ and $r_{bw}$.

\begin{claim}\label{claim:R1}
$\card{\cR_1}\ge \frac12 \left( \e(t)kn^{1/b} \right)^{r_{fw}} \cdot \left( \e(t)^2k^{b/(b-1)} \right)^{r_{bw}}$.
\end{claim}

\begin{proof}
We first show that at most $ ab + 2^{2ab}\de k^{b/(b-1)} $ choices are excluded for each vertex.  
Recall that $H=(V,E)$ is the graph induced by the currently selected edges. Note that at most $ab$ choices are excluded by the condition that the new vertex is not in $V$.  
Moreover, by \cref{lemma:linkbound} we have
\[
|L^{(1)}_{\cF}(E)| \le 2^{ab+\card{E}+1}\cdot \de k^{b/(b-1)} \le 2^{2ab} \de k^{b/(b-1)},
\]
as required.
Therefore, there are at least 
\[
\e(t)kn^{1/b} - \left(ab + 2^{2ab}\de k^{b/(b-1)} \right) \ge 2^{-1/r} \e(t)kn^{1/b}
\] 
choices for each forward vertex, where the last inequality holds since $k \le n^{(b-1)/b}$ and $ \de \ll \e(t)^2 $. Similarly, using the fact that $ \de \ll \e(t)^2 $, we find that there are at most
\[ 
\e(t)^2k^{b/(b-1)} - \left(ab + 2^{2ab}\de k^{b/(b-1)} \right) \ge 2^{-1/r} \e(t)^2 k^{b/(b-1)}
\] 
choices for each backward vertex. The claim now follows, as we choose $r_{fw}$ forward vertices and $r_{bw}$ backward vertices.
\end{proof}

For each $i \in [a-1]$, define $\cD_i$ to be the set of all $\bR_i \in \cR_i$ for which there are at least 
\[
\frac14 \e(t)^{t} k^{(t-1)b/(b-1)}
\]
paths $P \in \cQ(z_{i+1})$ with $E(P) \in L_\cF^{(t)}(E(\bR_i))$. (Here we view $\bR_i$ as a graph.)
We now deduce \cref{claim:R0} from the following two claims. The first shows that if the graph $\bR_i \in \cR_i$ satisfies $\bR_i \notin \cD_i$, then we have many choices for the path $P_{i+1}$ in Step 4.

\begin{claim}\label{claim:R2}
If $\bR_i \in \cR_i\setminus \cD_i$ for some $i\in [a-1]$, then 
\[
\card{\cR_{i+1}(\bR_i)} \ge \frac12 \cdot \e(t)^{t} k^{(t-1)b/(b-1)}.
\]
\end{claim}

The second states that $\card{\cD_i}$ is not too large.

\begin{claim}\label{claim:R3}
$\card{\cD_i} \le \tfrac12\card{\cR_i}$ for every $i\in [a-1]$. 
\end{claim}

\begin{proof}[Proof of \cref{claim:R0}]
From Claims \ref{claim:R2} and \ref{claim:R3}, we find
\[
\card{\cR_{i+1}} \ge \tfrac12 \cdot \e(t)^{t} k^{(t-1)b/(b-1)}\cdot \card{\cR_i\setminus \cD_i} \ge \tfrac14 \cdot \e(t)^{t} k^{(t-1)b/(b-1)}\cdot \card{\cR_i}
\]
for every $i\in [a-1]$. Combined with \cref{claim:R1}, we obtain
\begin{align*}
\card{\cR_a} &\ge \left(\tfrac14\e(t)^{t}k^{(t-1)b/(b-1)}\right)^{a-1}\cdot\tfrac12 ( \e(t)kn^{1/b})^{r_{fw}} \cdot (\e(t)^2k^{b/(b-1)})^{r_{bw}}\\
                   &=2^{-2a+1}\e(t)^{(a-1)t+r_{fw}+2r_{bw}}\cdot k^{(a-1)(t-1)b/(b-1)+r_{fw}+br_{bw}/(b-1)}n^{r_{fw}/b},
\end{align*}
By \eqref{eq:r}, one has
\begin{align*}
	(a-1)t+r_{fw}+2r_{bw} & \le (a-1)t+\left\{t+a(b-t)/2\right\}+2a(b-t)/2\\
	&=2ab-at-a(b-t)/2\le 2ab-at.
\end{align*}
Again from \eqref{eq:r} we find 
$r_{fw} \ge t+a(b-t)/2=b+(a-2)(b-t)/2 \ge b$. Together with the fact that $n\ge k^{b/(b-1)}$, this yields
\begin{align*}
k^{(a-1)(t-1)b/(b-1)+r_{fw}+br_{bw}/(b-1)}n^{r_{fw}/b} &\ge k^{(a-1)(t-1)b/(b-1)+r_{fw}+br_{bw}/(b-1)+(r_{fw}/b-1)\cdot b/(b-1)}n\\
&=k^{\left\{(a-1)(t-1)+r_{fw}+r_{bw}-1\right\}b/(b-1)}n=k^{ab}n,
\end{align*}
where the last equality follows from the formula $r_{fw}+r_{bw}=ab-(a-1)t$. Therefore, we get
\[
\card{\cR_a} \ge 2^{-2a+1}\e(t)^{2ab-at}k^{ab}n.
\]
As each copy of $\theta_{a,b}$ appears at most $a!$ times in $\cR_a$, we conclude 
\[
\card{\Theta} \ge  \frac{1}{a!}\card{\cR_a} \ge  \frac{1}{a!}2^{-2a+1}\e(t)^{2ab-at}k^{ab}n \ge \e(1)^{2ab}k^{ab}n
\]
for $\e(1)<\e(t) \ll 1$, 
as required.
\end{proof}

We end this section with the proofs of \cref{claim:R2,claim:R3}.

\begin{proof}[Proof of \cref{claim:R2}]
As $\card{\cQ(z_{i+1})}=\e(t)^tk^{(t-1)b/(b-1)}$, in order to prove the claim, it suffices to show $\card{\cQ(z_{i+1})\setminus \cR_{i+1}(\bR_i)} \le \frac12 \e(t)^tk^{(t-1)b/(b-1)}$. In other words, we wish to bound the number of paths in $\cQ(z_{i+1})$ which either contain a vertex of $V(\bR_i)\setminus \{x,z_{i+1}\}$, or fail to avoid $L_{\cF}(E(\bR_i))$.

By property \cref{lemma:bal-paths-cont-vertex} of \cref{thm:refined}, the number of paths in $ \cQ(z_{i+1})$ which contain a vertex of $V(\bR_i)\setminus \{x,z_{i+1}\}$ is at most
\begin{equation*}
ab \cdot b k^{(t-2)b/(b-1)} \le \e(t)^{t+1} k^{(t-1)b/(b-1)},
\end{equation*}
as $ k \ge k_0 \gg \e(t)^{-(t+1)} $. 

Now, let $\si \in L_\cF(E(\bR_i))$, and consider the paths in $\cQ(z_{i+1})$ that contain $\si$. We first deal with the case $1\le \card{\si} \le t-1$. According to \cref{lemma:bal-paths-cont-edges}, the number of paths in $\cQ(z_{i+1})$ containing $\si$ is at most 
\[
t^t k^{(t-\card{\si}-1)b/(b-1)}.
\]
Moreover, by \cref{lemma:linkbound}, we have 
\[
|L_\cF^{(|\si|)} (E(\bR_i))| \le 2^{2ab}( \de k^{b/(b-1)})^{\card{\si}}.
\]
Therefore, the number of paths in $\cQ(z_{i+1})$ which contain some $ \si \in L_\cF(E(\bR_i))$ with $1 \le \card \si \le t-1$ is at most
\begin{equation*}
(2b)^{2ab}\cdot \de k^{(t-1)b/(b-1)} \le \e(t)^{t+1}k^{(t-1)b/(b-1)},
\end{equation*}
as $\de \ll \e(t)^{t+1}$.

On the other hand, since $\bR_i \notin \cD_i$, there are at most \begin{equation*}%\label{eq:count3}
    \frac14 \e(t)^{t}k^{(t-1)b/(b-1)}
\end{equation*}
paths in $\cQ(z_{i+1})$ that contain some $\si \in L_{\cF}(E(\bR_i))$ with $\card{\si}=t$.

Summing these estimates gives the desired inequality
\[
\card{\cQ(z_{i+1})\setminus \cR_{i+1}(\bR_i)} \le \frac12 \cdot \e(t)^{t}k^{(t-1)b/(b-1)}. \qedhere
\]
\end{proof} 
 
\begin{proof}[Proof of \cref{claim:R3}.]
We proceed by induction on $i$. Let $i \in [a-1]$ and assume that the claim holds up to $i-1$ (no assumption is needed in case $i=1$). Thus, we have
\[
\card{\cR_j} \overset{\text{\cref{claim:R2}}}{\ge} \tfrac12 \cdot \e(t)^{t}k^{(t-1)b/(b-1)}\cdot\card{\cR_{j-1}\setminus\cD_{j-1}}\ge \tfrac14 \cdot \e(t)^{t}k^{(t-1)b/(b-1)}\cdot \card{\cR_{j-1}} 
\]
for every $2 \le j \le i$. Combined with \cref{claim:R1}, this gives
\begin{equation}\label{eq:claim4-R}
\card{\cR_i} \ge \left(\tfrac14 \e(t)^{t}k^{(t-1)b/(b-1)}\right)^{i-1}\cdot \tfrac12 \left( \e(t)kn^{1/b} \right)^{r_{fw}} \cdot \left( \e(t)^2k^{b/(b-1)} \right)^{r_{bw}}.
\end{equation}

We proceed in three steps. We first give an upper bound for the number of members of $\cD_i$ containing a given set of edges. This will be used in conjunction with \eqref{eq:claim4-R}. 

\begin{step}\label{step:1}
	Let $z\in B_t$, and let $J \subset E(G)$ be a forest of size $\card{J}= j \in [r+(i-1)(t-1)-1]$ which does not contain an $x$-$z$-path. Then there are at most
	\[m(j):=
	\begin{cases}
	(ab)^{2ab}\left(kn^{1/b}\right)^{r_{fw}-1} \left(k^{b/(b-1)}\right)^{r_{bw}+(i-1)(t-1)-j} & \text{if } 1 \le j \le r_{bw} +(i-1)(t-1)\\
	(ab)^{2ab}\left(kn^{1/b}\right)^{r+(i-1)(t-1)-j-1}  & \text{otherwise}
	\end{cases}
	\]
	$\bR_i \in \cR_i $ with $z_{i+1} = z$ and $J \subset E(\bR_i)$.
\end{step}

\noindent \textbf{Remark:} The key property about $m(j)$ is that it satisfies
\begin{equation}\label{eq:m(j)}
	m(j) \cdot ( \de k^{b/(b-1)} )^{j} \le \frac{\card{\cR_i}}{kn^{1/b}} \quad \text{for every} \enskip j \in [r+(i-1)(t-1)-1],
\end{equation}
due to \eqref{eq:claim4-R} and $ \de \ll \e(t)^{3ab}$. 

\begin{proof}[Proof of \cref{step:1}]
Note that we have at most $(ab)^{ab}$ choices for the positions of the edges of $J$ in $\bR_i$. Let's fix such a choice and count the corresponding $\bR_i$. More precisely, given a partition $J=J_1 \cup \ldots \cup J_i$, we shall bound from above the number of $\bR_i=(\bR_1,P_2, \ldots,P_i) \in \cR_i$ such that $J_1 \subset \bR_1$, $J_\ell \subset P_\ell$ for all $2 \le \ell \le i$, and $z_{i+1}=z$. 

We may assume $\card{J_\ell} \le t$ for every $2 \le \ell \le t$ (otherwise there is no such $\bR_i$). Let $I$ denote the set of all $\ell \in \{2,\ldots,i\}$ with $\card{J_{\ell}}=t$.
Note that $P_{\ell}=J_{\ell}$ for all $\ell \in I$, and so $\{z_{\ell}:\ell\in I \}$ is fixed. As $J$ contains neither an $x$-$z$-path nor a cycle, we may assume further that the subgraph induced by $J_1$ together with the fixed vertices $\{x,z\}\cup \{z_{\ell}:\ell\in I \}$ is a forest, in which these $\card{I}+2$ fixed vertices are in different components. It follows that at least $\card{J_1}+\card I + 2$ vertices of $\bR_1$ are fixed, and hence there are at most $r-\card{J_1}-\card{I}-1$ not-yet-chosen vertices in $\bR_1$ ($x$ is excluded). This shows
\[
r_1+r_2 \le r-\card{J_1}-\card{I}-1,
\]
where $r_1$ and $r_2$ denote the number of free forward vertices and free backward vertices respectively. In addition, as $z\in B_t$ is fixed, we must have
\[
r_1 \le r_{fw}-1.
\]

Note that we have at most $\e(t)kn^{1/b} \le kn^{1/b}$ choices for each forward vertex and at most $ \e (t)^2 k^{b/(b-1)} \le k^{b/(b-1)}$ choices for each backward vertex. Moreover, $P_{\ell}=J_{\ell}$ for all $\ell \in I$, and for each $\ell \in \{2,\ldots,i\}\setminus I$, there are at most $t^t k^{(t - \card{J_\ell} - 1) b/(b-1)}$ choices for $P_\ell$ by \cref{lemma:bal-paths-cont-edges}. Hence the number of $\bR_i=(\bR_1,P_2,\ldots,P_i)$ such that $(J_1,J_2,\ldots,J_i) \subset \bR_i$ is at most
\[
(kn^{1/b})^{r_1}\cdot (k^{b/(b-1)})^{r_2} \cdot\prod_{\ell \in \{2,\ldots,i\}\setminus I} t^{t} k^{(t-\card{J_{\ell}}-1)b/(b-1)} \le b^{ab}\cdot (kn^{1/b})^{r_1}\cdot (k^{b/(b-1)})^{r_3},
\]
where $r_3:=r_2+\sum_{\ell \in \{2,\ldots,i\}\setminus I}(t-\card{J_{\ell}}-1)$. 
To estimate the above expression, we note that
\begin{align*}
	r_1+r_3&=(r_1+r_2)+\sum_{\ell \in \{2,\ldots,i\}\setminus I}(t-\card{J_{\ell}}-1)\\
	&\le (r-\card{J_1}-\card{I}-1)+ \Big(|I|+\sum_{\ell \in \{2,\ldots,i\}}(t-\card{J_{\ell}}-1)\Big)\\
	&=r+(i-1)(t-1)-j-1,
\end{align*}
where the second line follows from the estimate $r_1+r_2 \le r-\card{J_1}-\card{I}-1$ and the definition of $I$, and the last equality holds since $|J_1|+\ldots+|J_{\ell}|=|J|=j$. 
Together with the inequalities $r_1 \le r_{fw}-1$ and $kn^{1/b} \ge k^{b/(b-1)}$, this implies that
the number of $\bR_i=(\bR_1,P_2,\ldots,P_i)$ with $(J_1,J_2,\ldots,J_i) \subset \bR_i$ is bounded from above by $b^{ab}\cdot (kn^{1/b})^{r_{fw}-1}\cdot (k^{b/(b-1)})^{r_{bw}+(i-1)(t-1)-j}$ in case $r_{fw}-1 \le r+(i-1)(t-1)-j-1$, and by $b^{ab}\cdot (kn^{1/b})^{r+(i-1)(t-1)-j-1}$ otherwise.

Putting everything together, we conclude that there are at most $m(j)$ choices for $\bR_i \in \cR_{i+1}$ with $z_{i+1}=z$ and $J \subset E(\bR_i)$.
\end{proof}

We shall use the inequality \eqref{eq:m(j)} in the proof of \cref{step:3} below.

\begin{step}\label{step:2}
	There exist $j \in [r+(i-1)(t-1)-1]$
	for which there are at least
	\[
	2^{-ab}\e(t)^{t}k^{(t-1)b/(b-1)} \cdot \frac{\card{\cD_i}}{ab\cdot m(j)}
	\]
	distinct pairs $(J,P)$ with the following properties:
	\begin{enumerate}[(a)]
		\item[\normalfont (a)] $P \in \cQ(z)$ for some $z \in B_t$,
		\item[\normalfont (b)] $J$ is a set of $j$ edges of $G$ disjoint from $E(P)$,
		\item[\normalfont (c)] $J \cup E(P) \in \cF$.
	\end{enumerate}
\end{step}

\begin{proof}
Recall that for each $\bR_i \in \cD_i$, there are at least $\frac14 \e(t)^{t}k^{(t-1)b/(b-1)}$ paths $P$ in $\cQ(z_{i+1})$ with $E(P)\in L^{(t)}_{\cF}(E(\bR_i))$. By the pigeonhole principle, it follows that for each $\bR_i \in \cD_i$, there exists a set $\emptyset \ne f(\bR_i) \subset E(\bR_i)$ such that there are at least 
\begin{equation}\label{eq:badpaths}
2^{-ab}\e(t)^{t}k^{(t-1)b/(b-1)}
\end{equation}
paths $P \in \cQ(z_{i+1})$, each of which is disjoint from $f(\bR_i)$ and with $f(\bR_i)\cup E(P) \in \cF$.
Note that $f(\bR_i)$ is a forest and does not contain an $x$-$z_{i+1}$-path (otherwise for every path $P \in \cQ(z_{i+1})$, $f(\bR_i) \cup E(P)$ contains a cycle and thus $f(\bR_i) \cup E(P) \not \in \cF$). In particular, it follows that $ \card{f(\bR_i)} \in [r+(i-1)(t-1)-1]$. By another application of the pigeonhole principle, there exists some $j \in   [r+(i-1)(t-1)-1]$ such that $\card{f(\bR_i)} = j$ for at least $\card{\cD_i}/ab$ choices of $\bR_i \in \cD_i$.

Now, define $\cJ$ to be the set of all pairs $(J,z)$ with $z \in B_t$, $\card J = j$ and $J = f(\bR_i)$ for some
$\bR_i \in \cD_i$ with $z_{i+1} = z$. We claim that $\card{\cJ} \ge \frac{\card{\cD_i}}{ab\cdot m(j)}$.
Indeed, there is such a pair $(f(\bR_i),z_{i+1})$ for each $\bR_i \in \cD_i$ with $ \card{f(\bR_i)} = j$, and we may have counted each pair $m(j)$ times, by the above discussion and \cref{step:1}.

Finally, for each $(J,z) \in \cJ$ choose some $\bR_i \in \cD_i$ with $f(\bR_i) = J$ and $z_{i+1} = z$. Recall that there are at least \cref{eq:badpaths} paths $P \in \cQ(z_{i+1})$ with 
$J\cup E(P) \in \cF$, each of which is disjoint from $J$. 
Since $P$ determines $z$, all such generated pairs $(J,P)$ are distinct, and hence the claim follows.
\end{proof}

We are now ready to show that $\card{\cD_i}$ is not too large.
\begin{step}\label{step:3}
$\card{\cD_i} \le \frac12 \card{\cR_i}$.
\end{step}
\begin{proof}
Let $N$ be the number of copies of $\theta_{a,b}$ in $G$ which contain an edge between $x$ and $B_1$. For each pair $(J,P)$ as in \cref{step:2}, we have $\card{J \cup E(P)} = j + t$ and $J\cup E(P) \in \cF$, giving
\[
d_\cH(J\cup E(P)) \ge \lfloor \De^{(j + t)}(\de,k,n) \rfloor \ge \frac12 \cdot\frac{\De^{(1)}(\de,k,n)}{\left( \de k^{b/(b-1)} \right)^{j+t-1}}.
\]
Thus, noting that each member of $\cH$ contains $J\cup E(P)$ for at most $2^{2ab}$ pairs $(J,P)$, it follows from \cref{step:2} that
\begin{align*}
	N&\ge  2^{-ab}\e(t)^{t}k^{(t-1)b/(b-1)} \cdot \frac{\card{\cD_i}}{ab\cdot m(j)} \cdot
	\frac{\De^{(1)}(\de,k,n)}{2^{2ab+1}\left( \de k^{b/(b-1)} \right)^{j+t-1}} \\
	&\ge \frac{\e(t)^{t}}{2^{4ab}\de^{t-1}} \cdot \frac{\card{\cD_i} \cdot \De^{(1)}(\de,k,n)}{m(j) \cdot \left( \de k^{b/(b-1)} \right)^{j}}\\
	&\ge 2kn^{1/b} \cdot \frac{\card{\cD_i} \cdot \De^{(1)}(\de,k,n)}{\card{\cR_i}},
\end{align*}
where the last inequality follows from \eqref{eq:m(j)}, and since $t\ge 2$ and $\de \ll \e(t)^{t}$.
Now, as $\card{B_1}\le kn^{1/b}$, there exists an edge $e \in E(G)$ with
\[
d_\cH(e) \ge 2 \cdot \frac{\card{\cD_i}\cdot \De^{(1)}(\de,k,n)}{\card{\cR_i}}.
\] 
Combined with \eqref{eq:degreebound2}, we get the desired inequality.
\end{proof}
This finishes the proof of \cref{claim:R3}.
\end{proof}

\section{Complete degenerate hypergraphs}
In this section we will prove \cref{thm:supersat-compl}.
We shall record the vertex partition of each copy of $K^{(r)}_{a_1,\ldots,a_r}$. Therefore, a collection $\cH$ of copies of $K^{(r)}_{a_1,\ldots,a_r}$ in an $r$-graph $G$ will be understood as a collection of ordered $r$-tuples $(A_1,\ldots,A_r)$ with $A_i \in \binom{V(G)}{a_i}$ for every $i\in [r]$, and with $G[A_1,\ldots,A_r] = K^{(r)}_{a_1,\ldots,a_r}$.

Given a tuple $(S_1,\ldots,S_r)$ of vertex sets such that $1\le \card{S_i} \le a_i$ for every $i\in [r]$ and $G[S_1,\ldots,S_r]$ is a complete $r$-partite $r$-graph, we define $d_{\cH}(S_1,\ldots,S_r)$ to be the number of members of $\cH$ containing $(S_1,\ldots,S_r)$, that is, 
\[
d_{\cH}(S_1,\ldots,S_r)=\card{\{(A_1,\ldots,A_r)\in \cH:(S_1,\ldots,S_r)\subset (A_1,\ldots,A_r)\}}.
\]
Let $\de>0$. For each $(b_1,\ldots,b_r)\in \N^r$ with $1\le b_i \le a_i$ for all $i\in [r]$, define
\begin{equation}\label{eq:D}
D^{(b_1,\ldots,b_r)}(\de,k,n)=\prod_{i=1}^{r}\left(\de k^{a_1\cdots a_{i-1}}n^{1-1/a_i\cdots a_{r-1}}\right)^{a_i-b_i},
\end{equation}
where $a_1\cdots a_{i-1}:=1$ if $i=1$ and $a_i\cdots a_{r-1}:=1$ if $i=r$. In particular, we have
\begin{equation} \label{eq:D(1,...,1)}
D^{(1,\ldots,1)}(\de,k,n)=\de^{a_1+\ldots+a_r-r}k^{a_1\cdots a_r-1}n^{a_1+\ldots + a_{r-1}-r+1/a_1\cdots a_{r-1}}.
\end{equation}

One can show that (see  \cref{sec:appendix-B}) \cref{thm:supersat-compl} is a consequence of the following result.
  
\begin{thm}\label{thm:supersat-r-partite}
	For every $2\le a_1\le \ldots \le a_r$, there exist constants $\de>0$ and $k_0\in\N$ such that the following holds for every $k\ge k_0$ and every $n\in \N$. Given an $r$-graph $G$ with $n$ vertices and $kn^{r-1/a_1\cdots a_{r-1}}$ edges, there exists a collection $\cH$ of copies of $K^{(r)}_{a_1,\ldots,a_r}$ in $G$, satisfying:
 \begin{enumerate}[\normalfont (a)]
		\item $\card{\cH}\ge \de^{a_1+\ldots+a_r} k^{a_1\cdots a_r}n^{a_1+\ldots + a_{r-1}}$, and
		\item $d_{\cH}(S_1,\ldots,S_r) \le D^{(\card{S_1},\ldots,\card{S_r})}(\de,k,n)$ for every $S_1,\ldots,S_r \subset V(G)$.
	\end{enumerate}	
\end{thm}

Fix now $2\le a_1\le \ldots \le a_r$. We shall need various constants in the proof of \cref{prop:supersat-complete} below, which we will define here for convenience. Informally, they will satisfy
	\[
	k_0 \gg K \gg 1\gg \e(1) \gg \e(2) \gg \ldots \gg \e(r)\gg \e(r+1)= \de >0.
	\]
	More precisely, we can set $\e(1)=1/2$, $\e(i+1)=\e(i)^{a_i}/(2^{2a_i+a_1\cdots a_i}a_i!)$ for each $1\le i \le r$, $\de=\e(r+1)$, $K=a_1\cdots a_{r}2^{a_1+\ldots+a_r+1}$ and $k_0=1/\de$.

Let $G$ be an $n$-vertex $r$-graph with $kn^{r-1/a_1\cdots a_{r-1}}$ edges, where $k\ge k_0$. Let $(S_1,\ldots,S_r)$ be an ordered $r$-tuple of vertex sets that satisfies $1\le \card{S_i} \le a_i$ for every $i\in [r]$, and with $G[S_1,\ldots,S_r]=K^{(r)}_{\card{S_1},\ldots,\card{S_r}}$. We say that $(S_1,\ldots,S_r)$ is {\em saturated} if 
\[
d_{\cH}(S_1,\ldots,S_r) \ge \lfloor D^{(\card{S_1},\ldots,\card{S_r})}(\de,k,n)\rfloor,
\]
and that $(S_1,\ldots,S_r)$ is {\em good} if it contains no saturated $r$-tuple. We say that $\cH$ is {\em good} if every $(A_1,\ldots,A_r) \in \cH$ is good.

\begin{prop}\label{prop:supersat-complete}
  	Suppose that $\cH$ is a good collection of copies of $K^{(r)}_{a_1,\ldots,a_r}$ in $G$ of size ${\card{\cH}\le \de^{a_1+\ldots+a_r} k^{a_1\cdots a_r}n^{a_1+\ldots + a_{r-1}}}$. Then, there exists a copy $(A_1,\ldots,A_r) \notin \cH$ of $K^{(r)}_{a_1,\ldots,a_r}$ in $G$ such that $\cH\cup \{(A_1,\ldots,A_r)\}$ is good. 	
\end{prop}

\begin{proof}
	Let $\cF$ denote the collection of saturated sets, i.e.\
	\[
	\cF=\{(S_1,\ldots,S_r):\emptyset \ne S_1,\ldots,S_r \subset V(G) \ \text{and} \ d_{\cH}(S_1,\ldots,S_r)=\lfloor D^{(\card{S_1},\ldots,\card{S_r})}(\de,k,n)\rfloor\}.
	\]
	 A simple double-counting argument shows that there are at most
	\[
	\frac{a_1\cdots a_r\cdot \card{\cH}}{\lfloor D^{(1,\ldots,1)}(\de,k,n) \rfloor} \ll kn^{r-1/a_1\cdots a_{r-1}}=e(G)
	\]
	saturated edges of $G$. Here we use the assumption that $\card{\cH} \le \de^{a_1+\ldots+a_r} k^{a_1\cdots a_r}n^{a_1+\ldots+a_{r-1}}$ and \eqref{eq:D(1,...,1)}. Thus by choosing a non-empty subhypergraph of $G$ if necessary (and weakening the bound on $e(G)$ slightly), we may assume that 
	\begin{equation}\label{eq:complete-good-edge}
		(\{v_1\},\ldots,\{v_r\}) \notin \cF \ \text{for every} \ \{v_1,\ldots,v_r\}\in E(G).
	\end{equation}
	
	For $S_1,\ldots,S_r \subset V(G)$ and $i\in [r]$, define
	$X_i(S_1,\ldots,S_r)$ to be the set consisting of all vertices $v\in V(G)\setminus (S_1\cup \ldots \cup S_r)$ so that $(S'_1,\ldots,S'_{i-1},S'_i\cup\{v\},S'_{i+1},\ldots,S'_r)\in \cF$ for some non-empty $S'_1\subset S_1, \ldots, S'_r\subset S_r$.
	
	\begin{claim}\label{claim:complete-saturated} Provided that $\card{S_i} \le a_i$ for each $i\in [r]$, we have
		\[
		\card{X_i(S_1,\ldots,S_r)}\le K\de k^{a_1\cdots a_{i-1}}n^{1-1/a_i\cdots a_{r-1}}.
		\]
	\end{claim}
	\begin{proof}
		For each tuple $(S'_1,\ldots,S'_r)$ with $\emptyset \ne S'_i \subset S_i$ for every $i\in [r]$, set
		\[
		\cJ(S'_1,\ldots,S'_r)=\{v \in V(G)\setminus (S_1\cup \ldots \cup S_r):(S'_1,\ldots,S'_i\cup\{v\},\ldots,S'_r)\in \cF\}.
		\]
		By the handshaking lemma and the definition of goodness, we obtain
		\[
		\frac{1}{a_1\cdots a_r}\cdot\sum_{v \in \cJ(S'_1,\ldots,S'_r)}d_{\cH}(S'_1,\ldots,S'_i\cup\{v\},\ldots,S'_r) \le d_{\cH}(S'_1,\ldots,S'_r) \le D^{(\card{S'_1},\ldots,\card{S'_r})}(\de,k,n),
		\]
		as each edge of $\cH$ is counted at most $a_1\cdots a_r$ times in the sum.
		Moreover,
		\[
		\sum_{v \in \cJ(S'_1,\ldots,S'_r)}d_{\cH}(S'_1,\ldots,S'_i\cup\{v\},\ldots,S'_r) \ge \card{\cJ(S'_1,\ldots,S'_r)}\cdot \lfloor D^{(\card{S'_1},\ldots,\card{S'_i}+1,\ldots,\card{S'_r})}(\de,k,n) \rfloor,
		\]
		by the definition of $\cJ$ and $\cF$. Hence $\card{\cJ(S'_1,\ldots,S'_r)}$ is bounded from above by
		\[
		(a_1\cdots a_r)\cdot \frac{D^{(\card{S'_1},\ldots,\card{S'_r})}(\de,k,n)}{\lfloor D^{(\card{S'_1},\ldots,\card{S'_i}+1,\ldots,\card{S'_r})}(\de,k,n) \rfloor} \le (2a_1\cdots a_r)\cdot\de k^{a_1\cdots a_{i-1}}n^{1-1/a_i\cdots a_{r-1}}.
		\]
		Finally, since the sets $\cJ(S'_1,\ldots,S'_r)$ cover $X_i(S_1,\ldots,S_r)$, we find that
		\[
		\card{X_i(S_1,\ldots,S_r)}\le \sum\card{\cJ(S'_1,\ldots,S'_r)} \le a_1\cdots a_r2^{\card{S_1}+\ldots+\card{S_r}+1}\cdot \de k^{a_1\cdots a_{i-1}}n^{1-1/a_i\cdots a_{r-1}}, 
		\]
		as desired.
	\end{proof}		
	
	We now show that there are at least $\de k^{a_1\cdots a_r}n^{a_1+\cdots+a_{r-1}}$ good $r$-tuples $(A_1,\ldots,A_r)$ with $\card{A_1}=a_1,\ldots,\card{A_r}=a_r$. From this, the proposition follows immediately, since at least one of these is not in $\cH$.
	
	\begin{claim}\label{claim:good-tuples}
		There are at least $\e(r+1)k^{a_1\cdots a_r}n^{a_1+\cdots+a_{r-1}}$ good $r$-tuples $(A_1,\ldots,A_r)$ with $\card{A_1}=a_1,\ldots,\card{A_r}=a_r$.
	\end{claim}
	\begin{proof}
		Let $i\in [r+1]$, and let $v_i,v_{i+1},\ldots,v_r$ be $r+1-i$ vertices of $G$ such that 
		\[
		d_G(v_i,\ldots,v_r) \ge 2^{i-r-1}\cdot kn^{i-1-1/a_1\cdots a_{r-1}}.
		\]
		We prove by induction on $i$ that there are at least 
		\[
		\e(i)n^{a_1+\ldots+a_{i-1}-(i-1)a_1\cdots a_{i-1}}\cdot d_G(v_i,\ldots,v_r)^{a_1\cdots a_{i-1}}
		\]
		good $r$-tuples $(A_1,\ldots,A_{i-1},\{v_i\},\ldots,\{v_r\})$ with $\card{A_1}=a_1,\ldots,\card{A_{i-1}}=a_{i-1}$. Here we set $a_1+\ldots +a_{i-1}:=0$ when $i=1$, and $d_G(v_i,\ldots,v_r):=kn^{r-1/a_1\cdots a_{r-1}}$ if $i=r+1$. It is easy to see that \cref{claim:good-tuples} follows from the case $i=r+1$.
		
		The base case $i=1$ is an immediate consequence of \eqref{eq:complete-good-edge}. Suppose, then, that the result holds for some $i \in [r+1]$. Fix $v_{i+1},\ldots,v_r\in V(G)$ with
		\begin{equation}\label{eq:largedegree}
			d_G(v_{i+1},\ldots,v_r)\ge 2^{i-r}\cdot kn^{i-1/a_1\cdots a_{r-1}}.
		\end{equation}
		Let $\cM$ denote the collection consisting of all $i$-tuples $(A_1,\ldots,A_{i-1},\{v\})$ with $v\in V(G)$, $\card{A_1}=a_1,\ldots,\card{A_{i-1}}=a_{i-1}$, and such that $(A_1,\ldots,A_{i-1},\{v\},\{v_{i+1}\},\ldots,\{v_r\})$ is good.
		
		\vspace{.3cm}\noindent \textbf{Subclaim 1:} $\card{\cM}\ge 2^{-a_1\cdots a_{i-1}}\e(i)\cdot n^{1+a_1+\ldots+a_{i-1}-ia_1\cdots a_{i-1}}\cdot d_G(v_{i+1},\ldots,v_r)^{a_1\cdots a_{i-1}}$.
		\begin{proof}[Proof of Subclaim 1]
			Set $X=\{v\in V(G):d_G(v,v_{i+1},\ldots,v_r)\ge \frac{1}{2n}\cdot d_G(v_{i+1},\ldots,v_r)\}$. Then
			\begin{equation*}
				\sum_{v\in X} d_G(v,v_{i+1},\ldots,v_r) \ge \frac12 d_G(v_{i+1},\ldots,v_r).
			\end{equation*}
			Fix a vertex $v\in X$. It follows from the definition of $X$ and the assumption \eqref{eq:largedegree} that
			$d_G(v,v_{i+1},\ldots,v_r) \ge 2^{i-r-1}\cdot kn^{i-1-1/a_1\cdots a_{r-1}}$.
		   Hence, by the induction hypothesis, $\cM$ contains at least
			$\e(i)n^{a_1+\ldots+a_{i-1}-(i-1)a_1\cdots a_{i-1}}d_G(v,v_{i+1},\ldots,v_r)^{a_1\cdots a_{i-1}}$ $i$-tuples of the form $(A_1,\ldots,A_{i-1},\{v\})$. \\
			Summing over all $v\in X$, and using Jensen's inequality give
			\begin{align*}
				\card{\cM} &\ge \e(i)n^{a_1+\ldots+a_{i-1}-(i-1)a_1\cdots a_{i-1}}\sum_{v\in X}d_G(v,v_{i+1},\ldots,v_r)^{a_1\cdots a_{i-1}}\\
				& \ge \e(i)n^{a_1+\ldots+a_{i-1}-(i-1)a_1\cdots a_{i-1}}\cdot |X|^{1-a_1\cdots a_{i-1}} \Big(\sum_{v\in X}d_G(v_{i+1},\ldots,v_r)\Big)^{a_1\cdots a_{i-1}}\\
				& \ge 2^{-a_1\cdots a_{i-1}}\e(i)\cdot n^{1+a_1+\ldots+a_{i-1}-ia_1\cdots a_{i-1}}\cdot d_G(v_{i+1},\ldots,v_r)^{a_1\cdots a_{i-1}}. \qedhere
				\end{align*}
		\end{proof}
		
		We now use Subclaim 1 to bound the number of good tuples $(A_1,\ldots,A_i,\{v_{i+1}\},\ldots,\{v_r\})$ with $\card{A_1}=a_1,\ldots,\card{A_i}=a_i$. For each $(i-1)$-tuple $\bA=(A_1,\ldots,A_{i-1})$, define
		\[
		\cM(\bA):=\{v\in V(G):(\bA,\{v\})\in \cM\}.
		\]
		Consider $(i-1)$-tuples 
		$\bA$ for which
		\begin{equation}\label{eq:typical-M}
		\card{\cM(\bA)} \ge  \frac12n^{-(a_1+\ldots+a_{i-1})}\card{\cM}.	
		\end{equation}
		
		\noindent \textbf{Subclaim 2:} Suppose $\bA$ satisfies \eqref{eq:typical-M}. Then there are $\frac{1}{2^{a_i}a_i!}\card{\cM(\bA)}^{a_i}$ sets $A_i\in \binom{\cM(\bA)}{a_i}$ so that $(\bA,A_i,\{v_{i+1}\},\ldots,\{v_r\})$ is good.
		\begin{proof}[Proof of Subclaim 2]
		From \eqref{eq:typical-M} and Subclaim 1, we see that
		\begin{equation}\label{eq:typical-M-cor}
		\card{\cM(\bA)} \ge 2^{(i-r-1)a_1\cdots a_{i-1}-1}\e(i)\cdot k^{a_1\cdots a_{i-1}}n^{1-1/a_i\cdots a_{r-1}}.
		\end{equation}
			For $j=1,\ldots,a_i$, we can pick an arbitrary vertex
			\[
			u_j\in \cM(\bA)\setminus\left(\{u_1,\ldots,u_{i-1}\}\cup X_i(\bA,\{u_1,\ldots,u_{j-1}\},\{v_i\},\ldots,\{v_r\})\right),
			\]	
			and let $A_i=\{u_1,\ldots,u_{a_i}\}$. By choice of $u_j$, the tuple $(A,\{u_1,\ldots,u_j\},\{v_{i+1}\},\ldots,\{v_r\})$ is good for every $j \in [a_i]$, and hence the $r$-tuple $(\bA,A_i,\{v_{i+1}\},\ldots,\{v_r\})$ is good. From \cref{claim:complete-saturated}, we deduce that the number of choices for each $u_j$ is  at least
			\[
			\card{\cM(\bA)}-\left(a_i+K\de k^{a_1\cdots a_{i-1}}n^{1-1/a_i\cdots a_{r-1}}\right) \overset{\eqref{eq:typical-M-cor}}{\ge} \card{\cM(\bA)}/2.
			\]
			Thus the total number of choices for $A_i$ is at least $\frac{1}{2^{a_i}a_i!}\card{\cM(\bA)}^{a_i}$.
		\end{proof}
		
		Finally, observe that 
		\[
		\sum_{\bA \enskip \text{satisfies} \enskip \eqref{eq:typical-M}}\card{\cM(\bA)}\ge \card{\cM}/2
		\]
		 and that there are at most $n^{a_1+\ldots+a_{i-1}}$ choices for $\bA=(A_1,\ldots,A_{i-1})$. Hence, by Subclaim 2 and convexity, the number of good $r$-tuples $(\bA,A_i,\{v_{i+1}\},\ldots,\{v_r\})$ is at least
		\begin{align*}
			\frac{1}{2^{a_i}a_i!}\sum_{\bA \enskip \text{satisfies} \enskip \eqref{eq:typical-M}}\card{\cM(\bA)}^{a_i} &\ge \frac{1}{2^{a_i}a_i!} \cdot \left(n^{a_1+\ldots+a_{i-1}}\right)^{1-a_i}(\card \cM/2)^{a_i}\\
			&\overset{\textrm{Subclaim 1}}{\ge} \e(i+1)n^{a_1+\ldots+a_i-ia_1\cdots a_i}\cdot d_G(v_{i+1},\ldots,v_r)^{a_1\cdots a_i}, 	
		\end{align*}
		completing the proof of \cref{claim:good-tuples}.	
	\end{proof}	
	
	This finishes our proof of Proposition \ref{prop:supersat-complete}.
\end{proof}

%%%%%%%%%%%%%%%%%%%%%%%%%%%%%%%%%%%%%%%%%%%%%%%%%%%%%%%%%%%%%%%%%%%%
%%%%%%%%%%%%%%%%%%%%%%%%%%%%%%%%%%%%%%%%%%%%%%%%%%%%%%%%%%%%%%%%%%%%
\section{The Tur\'an problem in random hypergraphs}

Balanced supersaturation theorems can be used to obtain results for the corresponding random hypergraph Tur\'an problem. This was done by Morris and Saxton in \cite{MorrisSaxton} for even cycles and complete bipartite graphs. For theta graphs and complete $r$-partite $r$-graphs, certain difficulties arise which we shall explain in this section.

Let $G^{(r)}(n,p)$ denote the random $r$-graph on $n$ vertices where each edge is present independently with probability $p$. For some fixed $r$-graph $H$, we denote by $\exn_r \left( G^{(r)}(n,p), H \right)$ the maximum number of edges of an $H$-free subgraph of $G^{(r)}(n,p)$. If $r=2$, we will drop the subscripts and superscripts. The following result provides lower bounds for all $p$. 

\begin{prop}\hfill
\label{prop:random-lowerbound}
\begin{enumerate}[(a)]
	\item Suppose that $\exn(n,\{C_3,C_4, \ldots, C_{2b}\}) = \Theta(n^{1+1/b})$. Then
	there is some positive constant $c = c(b)$ such that w.h.p.
	$\exn(G(n,p),C_{2b}) \geq cp^{1/b}n^{1+1/b}.$
	In particular, w.h.p. we have
	$\exn(G(n,p),\theta_{a,b}) \geq cp^{1/b}n^{1+1/b}$
	for every $a \geq 2$.
	\item Let $r \geq 2$ and $2 \leq a_1 \leq \ldots \leq a_r$, and suppose that $\exn_r(n,K_{a_1, \ldots, a_r}^{(r)})=\Theta \left(n^{r - 1/a_1 \cdots a_{r-1}}\right)$. Then, there is some positive constant $c=c(a_1,\ldots,a_r)$ such that w.h.p.\ 
	\[
	\exn_r(G^{(r)}(n,p),K_{a_1, \ldots, a_r}^{(r)}) \geq cp^{1 - 1/(a_1 \cdots a_{r-1})}n^{r - 1/a_1 \cdots a_{r-1}}.
	\]
\end{enumerate}	
	
\end{prop}

Morris and Saxton established Part (a) in \cite[Section 2.3]{MorrisSaxton}, while Part (b) can be obtained by adapting their construction.
In order to get good upper bound on the Tur\'an number $\exn_r \left( G^{(r)}(n,p), H \right)$, it is necessary to find the largest $\al >0$ for which $H$ is Erd\H os--Simonovits $\al$-good. To do so, one usually has to restrict the range of $k$ in \cref{def:Erdos--Simonovits-good}. Say that $H$ is Erd\H os--Simonovits $\al$-good for $m(n)$ up to $f(n)$, if the condition in \cref{def:Erdos--Simonovits-good} holds for every
$1\ll k \le f(n)$. Considering integers $a,b \ge 2$, one can easily deduce from \cref{thm:supersatcyclefree} that $\theta_{a,b}$ is  $\al$-Erd\H os--Simonovits good for $m(n) = n^{1+1/b}$ and $\al = 1/((a-1)b-1)$ up to $n^{(b-1)((a-1)b-1)/b(ab-1)}$. A standard application of the hypergraph container method (see \cite[Section 6]{MorrisSaxton}) then shows that,
for every $p \ge n^{-((a-1)(b-1)/(ab-1))}(\log n)^{2(a-1)b}$, we have w.h.p.\ $\exn(G(n,p),\theta_{a,b}) = O(p^{1/(a-1)b}n^{1+1/b})$. This matches the lower bound from \cref{prop:random-lowerbound} (a) only when $a=2$. Note that, in the case $a=2$, this recovers a result of Morris and Saxton \cite[Theorem 1.8]{MorrisSaxton}.

For complete $r$-partite $r$-graphs, the situation is very different. It is likely that \cref{thm:supersat-r-partite} is best possible and provides a matching upper bound for \cref{prop:random-lowerbound} (b) for large enough $p$. However, for $r >2$, finding the largest $\al$ so that $K_{a_1, \ldots, a_r}^{(r)}$ is $\al$-Erd\H os--Simonovits good up to some $f(n)$ from \cref{thm:supersat-r-partite} turns into a difficult optimisation problem we could not solve.

\noindent \textbf{Remark.} Recently, we have learned that Spiro and Verstra{\"e}te \cite[Theorem 1.3]{SpiroVerstraete} used our \cref{thm:supersat-r-partite} to essentially determine $\exn_r(G^{(r)}(n,p),K_{a_1, \ldots, a_r}^{(r)})$ in the regime $p \ge n^{-r/2}\log n$.

\section*{Acknowledgements}
This research was motivated by a series of lectures of Rob Morris at the Ramsey DocCourse programme in Prague 2016. The authors would like to thank him and the organisers of the course, Jaroslav Ne\v{s}et\v{r}il and Jan Hubi\v{c}ka. Furthermore, the authors would like to thank Peter Allen for his helpful comments on this text, and Joonkyung Lee for pointing out the reference \cite{Sidorenko93}. The authors appreciate suggestions from two anonymous referees which helped improve the exposition of the paper.

%%%%%%%%%%%%%%%%%%%%%%%%%%%%%%%%%%%%%%%%%%%%%%%%%%%%%
%%%%%%%%%%%%%%%%%%%%%%%%%%%%%%%%%%%%%%%%%%%%%%%%%%%%%

\setlength{\emergencystretch}{2em}
\providecommand{\bysame}{\leavevmode\hbox to3em{\hrulefill}\thinspace}
\providecommand{\MR}{\relax\ifhmode\unskip\space\fi MR }
\providecommand{\MRhref}[2]{%
	\href{http://www.ams.org/mathscinet-getitem?mr=#1}{#2}
}
\providecommand{\href}[2]{#2}

\appendix
\section{Proof of Proposition \ref{prop:supersat-count}}\label{sec:appendix-A}
The following proof is very similar to that of comparable statements given in \cite{MorrisSaxton,linear-cycles}.
We will make use of the hypergraph container method, developed in \cite{BaloghMorrisSamotij_Container,SaxtonThomason_Container}.
For an $s$-uniform hypergraph $\cH$, we define the \emph{maximum $j$-degree} $\De_j(\cH)$ of $\cH$ by
\[\De_j(\cH) = \max \{d_\cH(\sigma): \si \subset V(\cH) \text{ and } \card \si = j \}
\]
and the \emph{average degree} $d(\cH)$ of $\cH$ by $d(\cH)=s\card{E(\cH)}/\card{V(\cH)}$. Furthermore, for $\tau \in (0,1)$, the \emph{co-degree} function $\de(\cH,\tau)$ of $\cH$ is given by
\[
\de(\cH,\tau)=\frac{1}{d(\cH)}\sum_{j=2}^{s}\frac{\De_j(\cH)}{\tau^{j-1}}.
\]

\begin{thm}[see \cite{BaloghMorrisSamotij_Container,SaxtonThomason_Container}]\label{thm:containers}
	For each $s\in \N$, there exist positive constants $c_1=c_1(s)$ and $c_2=c_2(s)$ such that the following holds for all $N \in \N$. For each $0<\e<c_1$ and each $N$-vertex $s$-graph $\cH$, if $\tau \in (0,c_2)$ is such that $\de(\cH,\tau) \le \e$, then there exists a family $\cC$ of at most
	\begin{equation}\label{eq:number-containers}
		\exp\left(\frac{\tau\log(1/\tau)N}{\e}\right)
	\end{equation}
	subsets of $V(\cH)$ such that:
	\begin{itemize}
		\item[\normalfont (1)] for each independent set $I\subset V(\cH)$, there exists some $U\in\cC$ with $I\subset U$,
		\item[\normalfont (2)] $e(\cH[U]) \le \e e(\cH)$ for each container $U\in \cC$.
	\end{itemize}
\end{thm}

We shall establish \cref{prop:supersat-count} through iterated applications of the following consequence of Theorem \ref{thm:containers}.

\begin{prop}\label{prop:supersat-count-base}
	Suppose that an $r$-graph $H$ is Erd\H{o}s-Simonovits $\al$-good for $m=m(n)$. Then there exist positive constants $\e$ and $k_0$ such that the following holds for all $n,k \in \N$ with $k \geq k_0$. Given an $r$-graph $G$ on $[n]$ with $e(G)=k\cdot m(n)$, there exists a collection $\cC(G)$ of at most
	\[
	\exp\left(O(k^{-\al} \log k \cdot m(n))\right)
	\]
	subgraphs of $G$ satisfying:
	\begin{itemize}
		\item[\normalfont (1)] Every $H$-free subgraph of $G$ is a subgraph of some $U\in \cC$,
		\item[\normalfont (2)] $e(U) \le (1-\e)e(G)$ for every $U\in \cC$.
	\end{itemize}  
\end{prop}
\begin{proof}
	Since $H$ is Erd\H{o}s-Simonovits $\al$-good for $m=m(n)$, there exists a constant $C>0$ and a (non-empty) collection $\cH$ of copies of $H$ in $G$ such that
	\begin{equation}\label{eq:containers-1}
		d_{\cH}(\si) \le \frac{C\cdot \card{\cH}}{k^{(1+\al)(\card{\si}-1)}e(G)} \quad \text{for every $\si\subset E(G)$ with $1\le \card{\si}\le e(H)$}.
	\end{equation}
	We will now think of $\cH$ as a hypergraph whose vertex set is $E(G)$ and whose edges are the copies of $H$ in $\cH$.
	Set $1/\tau=\e^2k^{1+\al}$ and observe that, if $\e$ is sufficiently small,
	\begin{align*}
		\de(\cH,\tau)=\frac{1}{d(\cH)}\sum_{j=2}^{e(H)}\frac{\De_j(\cH)}{\tau^{j-1}}& \overset{\eqref{eq:containers-1}}{\le} \frac{1}{d(\cH)}\sum_{j=2}^{e(H)}\e^{2(j-1)}k^{(1+\al)(j-1)}\cdot \frac{C\cdot \card{\cH}}{k^{(1+\al)(j-1)}e(G)}\\
		&= \frac{C}{e(H)}\sum_{j=2}^{e(H)}\e^{2(j-1)} \le 2C\e^2/e(H) \le \e.
	\end{align*}	
	Using Theorem \ref{thm:containers}, we thus obtain a collection $\cC(G)$ of at most
	\[
	\exp\left(\frac{\tau\log(1/\tau)e(G)}{\e}\right) \le \exp\left(O(k^{-\al} \log k \cdot m(n))\right)
	\]
	subsets of $V(\cH)=E(G)$ such that:
	\begin{itemize}
		\item[\normalfont (1')] Every $H$-free subgraph of $G$ is a subgraph of some $U\in \cC(G)$, and
		\item[\normalfont (2')] $e(\cH[U]) \le \e e(\cH)$ for all $U\in \cC(G)$.
	\end{itemize}
	The only thing that remains to prove is that $e(U)\le (1-\e)e(G)$ for every $U\in \cC$. Consider an arbitrary container $U\in \cC$. From \eqref{eq:containers-1} we find
	\[
	\card{E(\cH)\setminus E(\cH[U])} \le \card{V(\cH)\setminus U}\cdot \frac{C\cdot e(\cH)}{e(G)}.
	\]
	On the other hand, it follows from condition (2') that $\card{E(\cH)\setminus E(\cH[U])} \ge (1-\e)e(\cH)$. Hence $\card{V(\cH)\setminus U} \ge (1-\e)e(G)/C \ge \e e(G)$, as desired.
\end{proof}	

We are now ready to prove Proposition \ref{prop:supersat-count}.

\begin{proof}[Proof of Proposition \ref{prop:supersat-count}]	
	We wish to estimate the number of $H$-free subgraphs of $K^{(r)}_n$. We define a sequence $\{k(i)\}_{i=1}^{t}$ of positive reals, and a sequence $\{\cC_i\}_{i=1}^{t}$ of families of $r$-graphs as follows. Let $\e$ and $k_0$ be positive constants given by Proposition \ref{prop:supersat-count-base}. We set $k(1)=\binom{n}{r}/m(n)$ and define $k(i)=(1-\e)k(i-1)$, with $k(t)$ being the first term of this sequence to satisfy $k(t) \le k_0$. We take $\cC_0=\{K^{(r)}_n\}$, and for $1 \le i \le t$, we obtain $\cF_i$ from $\cC_{i-1}$ by replacing each $r$-graph $G\in \cC_{i-1}$ for which $e(G) \ge k(i)\cdot m(n)$ by the collection $\cC(G)$ of its subgraphs guaranteed by Proposition \ref{prop:supersat-count-base}.
	
	Let $\cC=\cC_t$. Clearly, every $H$-free $r$-graph on $[n]$ is contained in some $G\in \cC$. Moreover, $e(G) \le k_0\cdot m(n)$ for every $G\in \cC$. Finally, from \eqref{eq:number-containers} we see that
	\[
	\card{\cF} \le \exp\left(\sum_{i=1}^{t} O(1) \cdot k(i)^{-\al} \log k(i) \cdot m(n)\right) \le \exp\left(O(1) \cdot k_0^{-\al} \log k_0 \cdot m(n)\right).
	\]
	
	Therefore, the number of $H$-free $r$-graphs on $[n]$ is at most
	\[
	\sum_{G\in \cC}2^{e(G)} \le \card{\cC}2^{k_0\cdot m(n)} \le \exp\left(O(1) \cdot k_0^{-\al} \log k_0 \cdot m(n) + k_0\cdot m(n)\right)=2^{O(m(n))}.
	\]   
	Finally, given any $ \de >0 $, the number of $H$-free $r$-graphs with $n$ vertices and less than $m(n)/k_0^3$ edges is bounded from above by
	\[
	\sum_{G\in \cC}\sum_{i=1}^{m(n)/k_0^3}\binom{e(G)}{i} \le \card{\cC}2^{m(n)/k_0} \le \exp\left(k_0^{-\al/2}m(n)+m(n)/k_0\right) \leq 2^{\de\cdot m(n)}
	\]
	if $k_0$ is large enough.
\end{proof}

\section{\texorpdfstring{Deriving \cref{thm:supersat-compl} from \cref{thm:supersat-r-partite}}{Deriving the main supersaturation theorem}}
\label{sec:appendix-B}
In this section we will deduce \cref{thm:supersat-compl} from \cref{thm:supersat-r-partite}.
Obviously, property (i) in \cref{thm:supersat-compl} follows from property (a) in \cref{thm:supersat-r-partite}. It remains to verify property (ii) in \cref{thm:supersat-r-partite}. 
Without loss of generality we can assume that $\sigma$ is a complete $r$-partite $r$-graph. In light of property (b), our task becomes to justify that
\begin{equation}\label{eq:est-D}	
D^{(b_1,\ldots,b_r)}(\de,k,n) \le \frac{C\cdot |\cH|}{k^{(1+\al)(b_1\cdots b_r-1)}e(G)}
\end{equation}
for every $(b_1,\ldots, b_r) \in \mathbb{N}^{r}$ satisfying $b_1 \le a_1, \ldots, b_r \le a_r$. From \eqref{eq:D}, we see that
\begin{equation*}
D^{(b_1,\ldots,b_r)}(\de,k,n) =O(1) \cdot k^{\be}n^{\ga},
\end{equation*}
where $\be=\sum\limits_{i=1}^{r}(a_i-b_i)a_1\cdots a_{i-1}$ and $\ga=\sum\limits_{i=1}^{r-1}(a_i-b_i)(1-1/a_i\cdots a_{r-1})$.
On the other hand, by the assumptions on $G$ and $\cH$, we obtain
\begin{equation*}
\frac{|\cH|}{k^{(1+\al)(b_1\cdots b_r-1)}e(G)}=\Omega(1)\cdot k^{\be'}n^{\ga'}
\end{equation*}
where $\be'=a_1\cdots a_r-(1+\al)(b_1\cdots b_r-1)-1$ and $\ga'=a_1+\ldots +a_{r-1}-r+1/a_1\cdots a_{r-1}$. As $\ga'=\sum\limits_{i=1}^{r-1}(a_i-1)(1-1/a_i\cdots a_{r-1})$, we get
\begin{equation}\label{eq:ga}
	\ga'-\ga=\sum\limits_{i=1}^{r-1}(b_i-1)(1-1/a_i\cdots a_{r-1}) \ge 0.
\end{equation}
We next show that
\begin{equation}\label{eq:be+ga-1}	
	(\be'-\be)+a_1\cdots a_{r-1} \cdot (\ga'-\ga) \ge 0.
\end{equation}
From \eqref{eq:ga}, \eqref{eq:be+ga-1} and the fact that $n \ge k^{a_1\cdots a_{r-1}}$, we find
\begin{align*}
	k^{\be'}n^{\ga'} &= k^{\be}n^{\ga} \cdot (k^{-a_1\cdots a_{r-1}}n)^{\ga'-\ga} \cdot k^{(\be'-\be)+a_1\cdots a_{r-1} \cdot (\ga'-\ga) } \ge k^{\be} n^{\ga},
\end{align*}
implying \eqref{eq:est-D} for $C$ sufficiently large.

In the remainder of this section, we shall justify \eqref{eq:be+ga-1}.
Observe that
\begin{align*}
	\be'&=(a_1\cdots a_r-1)-(b_1\cdots b_r-1)-(b_1\cdots b_r-1)\al\\
	&=\sum_{i=1}^{r}(a_i-1)a_1\cdots a_{i-1}-\sum_{i=1}^{r}(b_i-1)b_1\cdots b_{i-1}-(b_1\cdots b_r-1)\al.
\end{align*}
From this it follows that
\begin{align*}
	\be'-\be &=\sum_{i=1}^{r}(b_i-1)(a_1\cdots a_{i-1}-b_1\cdots b_{i-1})-(b_1\cdots b_r-1)\al.
\end{align*}
Combined with \eqref{eq:ga}, this yields
\begin{align}\label{eq:be+ga-2}
\notag	(\be'-\be)+a_1\cdots a_{r-1} \cdot (\ga'-\ga) &=
	\sum_{i=1}^{r}(b_i-1)(a_1\cdots a_{i-1}-b_1\cdots b_{i-1})-(b_1\cdots b_r-1)\al \\
	& \quad + a_1\cdots a_{r-1} \cdot \sum_{i=1}^{r-1}(b_i-1)(1-1/a_i\cdots a_{r-1}).
\end{align}
If $b_1=\ldots =b_{r-1}=1$, then the RHS of \eqref{eq:be+ga-2} is at least $(b_r-1)(a_1\cdots a_{r-1}-1)-(b_r-1)\al \ge 0$. So \eqref{eq:be+ga-1} is valid in this case.

Now suppose $(b_1,\ldots,b_{r-1})\ne (1,\ldots,1)$. Since $1\le b_i \le a_i$ for every $i \in [r]$, we can bound the RHS of \eqref{eq:be+ga-2} from below by
\begin{align*}
	-(b_1\cdots b_r-1)\al + a_1\cdots a_{r-1} \cdot \sum_{i=1}^{r-1}(b_i-1)(1-1/a_i\cdots a_{r-1}) &\ge -(b_1\cdots b_r-1)\al +1\\
	&\ge 0,
\end{align*}
where in the second inequality we used the fact that $a_1\cdots a_{r-1} \cdot \sum\limits_{i=1}^{r-1}(b_i-1)(1-1/a_i\cdots a_{r-1})$ is a positive integer when $(b_1,\ldots,b_{r-1})\ne (1,\ldots,1)$, and in the last inequality we estimated $\al=\frac{1}{a_1\cdots a_r-1} \le \frac{1}{b_1\cdot b_r-1}$. Hence \eqref{eq:be+ga-1} is true in this case as well. This completes our proof.
\end{document}